\documentclass [12pt,a4paper,reqno]{amsart}
 \input amssymb.sty
 \input xy
\xyoption{all}

\newcommand{\etype}[1]{\renewcommand{\labelenumi}{(#1{enumi})}}
\def\eroman{\etype{\roman}}

\def\pSkip{\vskip 1.5mm \noindent}
\newcommand{\ds}[1]{\ {#1} \ }
\newcommand{\dss}[1]{\quad {#1} \quad }

\def\sm{\setminus}
\def\00{ \{ 0 \}}
\def\ptM{M}
\def\hlf{1/2 }

\def\veps{\varepsilon}
\def\vrp{\varphi}
\def\brvrp{\overline{\vrp}}

\def\tlV{\widetilde V}

\def\Rig{\operatorname{Rig}}
\def\Quad{\operatorname{Quad}}
\def\STR{\operatorname{STR}}
\def\dv{{\operatorname{div}}}

\def\X1{X_1}
\def\Y1{Y_1}

\def\nul{<_\nu}

\def\nule{\leq_\nu}

\def\tlb{\tilde b}
\def\tlq{\tilde q}

\newcommand{\trn}[1]{{}^{\operatorname {t}}{#1}}
\def\tT{\mathcal T}
\def\tG{\mathcal G}

\def\tL{\mathcal L}

\def\sq{{\operatorname{sq}}}

\def\Rig{\operatorname{Rig}}
\def\Quad{\operatorname{Quad}}

\def\QL{\operatorname{QL}}

\def\N{\mathbb N}
\def\Z{\mathbb Z}
\def\R{\mathbb R}
\def\Q{\mathbb Q}

\def\mfo{\mathfrak o}

\def\ql{q_{\QL}}

\def\elm{\xi}
\def\rg{\rho}

\def\al{\alpha}
\def\bt{\beta}
\def\gm{\gamma}
\def\sig{\sigma}

\newtheorem{thm}{Theorem} [section]
\newtheorem*{thm*}{Theorem}
\newtheorem*{prop*}{Proposition}
\newtheorem{cor}[thm]{Corollary}
\newtheorem{lem}[thm]{Lemma}
\newtheorem{lemma}[thm]{Lemma}
\newtheorem{prop}[thm]{Proposition}

\newtheorem{convention}[thm]{Convention}
\newtheorem*{claim*} {Claim}
\newtheorem*{theorem4.5'} {Theorem 4.5$'$}
\newtheorem{acknowledgment*}[thm] {Acknowledgment}

\newtheorem{exampl}[thm]{Example}

\newtheorem{example}[thm]{Example}
\newtheorem{examp}[thm]{Example}

\newtheorem*{exampleA*}{Example A}
\newtheorem*{exampleB*}{Example B}
 \newtheorem{remark}[thm]{Remark}
  \newtheorem{remarks}[thm]{Remarks}
 \newtheorem*{remark*}{Remark}
 \newtheorem{defn}[thm]{Definition}

\newtheorem{terminology}[thm]{Terminology}
\newtheorem{schol}[thm]{Scholium}
\newtheorem{notation}[thm]{Notation}
\newtheorem{notations}[thm]{Notations}

\newtheorem*{notation*} {Notation}
\newtheorem*{comment*} {Comment}

\newtheorem{rem}[thm]{Remark}

\newcommand{\thmref}[1]{Theorem~\ref{#1}}
 \newcommand{\propref}[1]{Proposition~\ref{#1}}
\newcommand{\secref}[1]{\S\ref{#1}}

\newcommand{\lemref}[1]{Lemma~\ref{#1}}
\newcommand{\corref}[1]{Corollary~\ref{#1}}

\newcommand\thmreff[2]{\pSkip\textbf{Theorem #1. }\emph{#2}}
\newcommand\propreff[2]{\pSkip\textbf{Proposition #1. }\emph{#2}\pSkip}

 \renewcommand{\sectionmark}[1]{}

\newcommand{\bfem}[1]{\textbf{#1}}

\newcommand{\lm}{\lambda}

\renewcommand{\a}{\alpha}

\def\|{\ds |}

\begin{document}

\title[Supertropical  Quadratic Forms I]
{Supertropical  Quadratic Forms I}
\author[Z. Izhakian]{Zur Izhakian}
\address{Institute  of Mathematics,
 University of Aberdeen, AB24 3UE,
Aberdeen,  UK.}
    \email{zzur@abdn.ac.uk; zzur@math.biu.ac.il}
%\address{ \vskip -6mm CNRS et Universit�e Denis Diderot
%(Paris 7), 175, rue du Chevaleret 75013 Paris, France}
%\email{zzur@math.biu.ac.il}
\author[M. Knebusch]{Manfred Knebusch}
\address{Department of Mathematics,
NWF-I Mathematik, Universit\"at Regensburg 93040 Regensburg,
Germany} \email{manfred.knebusch@mathematik.uni-regensburg.de}
\author[L. Rowen]{Louis Rowen}
 \address{Department of Mathematics,
 Bar-Ilan University,  Ramat-Gan 52900, Israel}
 \email{rowen@math.biu.ac.il}

%\thanks{The research of the first and third authors was supported  by the
% Israel Science Foundation (grant No.  448/09).}

%\thanks{The research of the second author was supported in part by
% the Gelbart Institute at
%Bar-Ilan University, the Minerva Foundation at Tel-Aviv
%University, the Department of Mathematics   of Bar-Ilan
%University, and the Emmy Noether Institute at Bar-Ilan
%University.}

\thanks{The authors are grateful to the anonymous referee for the helpful comments and suggestions on an earlier version.}

%******************************* AMS classification ***********************
%\subjclass{??? Primary 12K10, 13B25; Secondary 51M20}
%\subjclass[2010]  {Primary: 13A18, 13F30, 16W60, 16Y60; Secondary:
%03G10, 06B23, 12K10,   14T05}

%******************************* AMS classification ***********************

%******************************* date *************************************
%\date{\today}

%******************************* keywords *********************************


%******************************* file name *********************************

% \thanks{\noindent \underline{\hskip 3cm } \\ File name: \jobname}

%******************************* abstract *********************************

\begin{abstract}
We initiate the theory of a quadratic form $q$ over a semiring $R$,
with a view to study tropical linear algebra. As customary, one can
write
$$ q(x+y) = q(x) + q(y)+ b(x,y),$$ where $b$ is a \textbf{companion
bilinear form}. In contrast to the classical theory of quadratic
forms over a field, the companion bilinear form need not be uniquely
defined. Nevertheless, $q$ can always be written as a sum of
quadratic forms $q = \ql + \rg,$ where $\ql$ is \textbf{quasilinear}
in the sense that $\ql(x+y) = \ql(x) + \ql(y),$ and $\rg$ is
\textbf{rigid} in the sense that it has a unique companion.
%We obtain an explicit classification
%of such decompositions and their companions.
In case that $R$ is   supertropical, we obtain an explicit
classification of these decompositions $q = \ql + \rg $ and of all
companions $b$ of $q$, and see how this relates to the
tropicalization procedure.

%
%As an application to tropical geometry, given a quadratic form $q:
%V \to R$ on a free module $V$ over a commutative ring~$R$ and a
%supervaluation $\vrp: R \to U$ with values in a supertropical
%semiring \cite{IzhakianKnebuschRowen2009Valuation}, we define  --
%after choosing a base $\tL=(v_i \ds |i\in I)$ of $V$ --   a
%quadratic form $q^\vrp: U^{(I)}\to U$ on the free module $U^{(I)}$
%over the semiring $U$.  The analysis of quadratic forms over a
%supertropical semiring enables one to measure  the ``position'' of~$q$ with respect to $\tL$ via~$\vrp$.
%%We obtain complete, explicit
%%results if $U$ is a ``supersemifield'' (Definition
%%\ref{defn:I.4.1} below).
\end{abstract}

\maketitle

\setcounter{tocdepth}{1}
\tableofcontents

\numberwithin{equation}{section}
%\newpage
% \vskip 10mm
\section*{\qquad Introduction and basic notions}

 One of the major areas in tropical
mathematics is tropical linear algebra, as indicated in
\cite{ABG,AGG,Butkovic} and the volumes \cite{LitSer1},\cite{LitSer} of the
Litvinov--Sergeev conferences. This essentially is the study of
vectors in $F^{(n)},$ where $F$ is a semifield. In
\cite{IzhakianRowen2008Matrices}--\cite{IzhakianRowen2010MatrixIII}
and \cite{IzhakianKnebuschRowen2010LinearAlg} the theory of
supertropical matrices and linear algebra was developed, and in
particular bilinear forms were utilized in
\cite{IzhakianKnebuschRowen2010LinearAlg} in order to study
orthogonality in supertropical vector spaces. This leads one
naturally to quadratic forms, the subject of this paper.

Our objective in this paper is to lay a general foundation of the
theory of quadratic forms over semirings,
 in particular, semifields; in other words,   the scalars are not required to have negation. Examples of interest include the positive rational numbers, more generally the set of sums of squares in a field, and rational functions with positive coefficients.
 Our main motivation is tropical (and supertropical) mathematics,  as described e.g. in \cite{IMS},
\cite{MS}, which is usually done over the max-plus algebra
 over $(\R,+)$ or $(\Q,+)$. Tropical mathematics has attracted
 considerable interest in the last two decades, and as in the
 classical case, quadratic form theory enriches the study of
 tropical linear algebra. Although the max-plus algebra
 could be viewed as an ordered group, the semiring approach enables
 us to translate the appropriate algebraic structure to the tropical
 framework, as indicated in \cite{MS}.

  Recall that a \bfem{semiring}~$R$ is a set~$R$ equipped with addition and multiplication,
  such that both $(R,+)$ and $(R,\cdot)$ are abelian monoids\footnote{A monoid means a semigroup that has a neutral element.} with elements $0=0_R$ and $1=1_R$ respectively, and multiplication distributes over addition in the usual way.
  In other words, $R$ satisfies all the properties of a commutative ring except the existence of negation under addition. We call
  a semiring $R$ a \bfem{semifield}, if every nonzero element of $R$ is invertible;
  hence $R\setminus\{0\}$ is an abelian group. A semiring $R$ is
  \bfem{bipotent} (often called ``selective'' in the literature) if $a+b \in \{a,b\}$ for all $a,b \in R$.

As in the classical theory,
 one often wants to consider bilinear forms over modules over an arbitrary  semiring $R$.
A \bfem{module} $V$ over $R$ is an abelian monoid $(V,+)$ equipped
with a scalar multiplication $R\times V\to V,$ $(a,v)\mapsto av,$
such that exactly the same axioms hold as customary for modules if
$R$ is a ring:

$a_1(bv)=(ab)v,$ $a_1(v+w)=a_1v+a_2w,$ $(a_1+a_2)v=a_1v+a_2v,$
$1_R\cdot v=v,$ $0_R\cdot v=0_V=a_1\cdot 0_V$ for all $a_1,a_2\in
R,$ $v,w\in V.$

 Most often we write $0$ for both $0_V$ and $0_R$ and
1 for $1_R.$

We call an $R$-module $V$ \bfem{free}, if there exists a family
$(\veps_i \ds|i\in I)$ in $V$ such that every $x\in V$ has a unique
presentation $x=\sum\limits_{i\in I} x_i\veps_i$ with scalars
$x_i\in R$ and only finitely many $x_i $ nonzero, and we call
$(\veps_i \ds |i\in I)$ a \bfem{base} of the $R$-module $V.$ The
obvious example is $V = R^{(n)}$, with the classical  base. In fact,
any free module with a base of $n$ elements is obviously isomorphic
to $R^{(n)}$, under the map $\sum_{i=1}^n x_i\veps_i\mapsto
(x_1, \dots, x_n).$

In contrast to the classical case of vector spaces over a field,
modules over a bipotent semifield need not be free, but as
consolation, if they are free, bases  are unique up to scalar
multiplication, cf.~\thmref{uniq} below. \pSkip

 As in the classical
theory, we are led naturally to our main notion of this paper:

 \begin{defn}\label{defn:I.1.1} For any   module $V$  over a semiring $R$,   a \bfem{quadratic form
on} $V$ is a function $q: V\to R$
with\begin{equation}\label{eq:I.1.1} q(ax)=a^2q(x)\end{equation}
 for any $a\in R,$ $x\in V,$ together with a
symmetric bilinear form $b: V\times V\to R$ (not necessarily
uniquely determined by $q$) such that for any $ x,y\in V$
\begin{equation}\label{eq:I.1.2} q(x+y)=q(x)+q(y)+b(x,y).\end{equation}
 Every such bilinear form $b$ will be called a
\bfem{companion} of $q$, and the pair $(q,b)$ will be called a
\bfem{quadratic pair} on $V.$ \end{defn}
 When $R$ is a
ring, then $q$ has just one companion, namely,
$b(x,y):=q(x+y)-q(x)-q(y),$ but if $R$ is a semiring which cannot be
embedded into a ring, this usually is wrong, and much of this paper
is concerned with investigating these companions.
%
%Just as in the classical theory, any bilinear form $b: V \times V
%\to R$ gives rise to a quadratic form $Q: V \to R$ given by $Q(v) =
%b(v,v),$ and much of the theory of bilinear forms can be
%encapsulated in the study of $Q$.

These semirings are closely related to (totally) ordered monoids with an absorbing element~$0$. Any such  monoid gives rise to the familiar \bfem{max-plus algebra}, whose multiplication is the original monoid operation, and whose addition is the maximum in the given ordering. The ensuing algebraic structure is that of a (commutative) \bfem{bipotent semiring}. Conversely, in ``logarithmic'' notation customarily used in tropical geometry, bipotent semirings appear as ordered additive monoids with absorbing element $-\infty.$ The primordial object here is the bipotent semifield $T(\R)=\R\cup\{-\infty\},$ cf. e.g. \cite[\S1.5]{IMS}.

% Any ordered group $M$ gives rise to the familiar (bipotent) \bfem{max-plus
%algebra}, whose multiplication is the original group operation, and
%whose addition is the maximum in the given ordering; this is a
%(commutative) bipotent semifield, whose zero element is
%$\{-\infty\}$ and whose unit element is the neutral element of $M$.
%
%Conversely, in the notation customarily used in tropical geometry,
%bipotent semifields appear as ordered groups with absorbing element
%$-\infty.$ Thus the primordial object here is the bipotent semifield
%$T(\R)=\R\cup\{-\infty\},$ cf. e.g. \cite[\S1.5]{IMS}.

It turns out that one can delve deeper into the underlying algebraic
structure by taking the so-called \bfem{supertropical} cover of
$T(\R)$. The first author  \cite{zur05TropicalAlgebra}  introduced a
cover of~$T(\R),$ graded by the multiplicative
monoid~$(\Z_2,\cdot)$, which was dubbed the \textit{extended
tropical arithmetic}. Then, in \cite{IzhakianRowen2007SuperTropical}
and \cite{IzhakianRowen2008Matrices}, this structure was amplified
to the notion of a \bfem{supertropical semiring}. A supertropical
semiring~$R$ is a semiring equipped with a ``ghost map'' $\nu=\nu_R:
R\to R$, which respects addition and multiplication and is
idempotent, i.e., $\nu\circ\nu=\nu.$ Moreover, in this semiring
$a+a=\nu(a)$ for every $a\in R$ (cf.
\cite[\S3]{IzhakianKnebuschRowen2009Valuation}). This replaces the
property $a+a=a$ taking place in the usual max-plus (or min-plus)
arithmetic. We write $a ^\nu$ for $\nu(a).$ We call $a ^\nu$ the
``\bfem{ghost}'' of $a,$ and we call the non-ghost elements of~$R$
``\bfem{tangible}''. (The element 0 is regarded both as tangible and
ghost.) $R$ then carries a multiplicative idempotent $e=1+1 = 1
^\nu$ such that $a ^\nu=ea$ for every $a\in R.$ The image $eR$ of
the ghost map, called the \bfem{ghost ideal} of $R,$ is itself a
bipotent semiring.

In this paper we are concerned primarily with quadratic forms on
$F^{(n)}$ where $F$ is a tangible supersemifield. Nevertheless, at
times one should  also admit certain semirings  related to
supertropical semirings, but which themselves are not supertropical,
e.g., polynomial function semirings in any number of variables over
a supertropical semiring, cf.
\cite[\S4]{IzhakianRowen2007SuperTropical}. Thus, we formulate
results in this generality, with the understanding that often  $V=
F^{(n)}$.

%
%Quadratic form theory over a semiring in general is an arid area.
% But from Chapter 4 on we focus
%on a rather special class of semirings, the so-called
%``\emph{supertropical semirings''}, which have been designed to
%enrich the algebraic toolbox for working in tropical mathematics, in
%particular, tropical geometry, as described e.g. in \cite{IMS},
%\cite{MS}.
%
%These semirings are closely related to (totally) ordered monoids.

 Here is an overview of the main contents
of this paper. The basic notation and definitions about quadratic
pairs and forms are provided in \S\ref{sec:I.1}, including the
``companion'' bilinear form. An important special case treated in
\S\ref{quasilin}:
 We call a quadratic form $q$ \bfem{quasilinear} if  $ q(x+y) = q(x) + q(y)$  for all $x,y \in V.$
Quasilinear forms over a supertropical semifield have been
considered in \cite[\S5]{IKR-LinAlg2}.

As mentioned earlier, a key digression from the classical theory is
that the companion bilinear form need not be unique. Consequently,
we call~$q$ \bfem{rigid} if $q$ has only one companion. The basic
properties of rigidity are given in \S\ref{sec:I.2}.

 \thmreff{\ref{thm:I.2.13}}{For $V$  free with base
$(\veps_i \ds | i \in I )$, a quadratic form $q$ on~$V$ is rigid iff
$q(\veps_i) = 0$ for all $i \in I$. }

%\smallskip

This is a fact of central importance for the whole paper.

 \propreff{\ref{prop:I.5.1}}{Any quadratic form $q$ can always be written as a sum of quadratic
forms $q = q_{\QL} + \rg,$ where $q_{\QL}$ is  quasilinear (and
uniquely determined) and $\rg$ is  rigid (but is not unique).} We
call this a \bfem{quasilinear-rigid decomposition} of $q$. Such
decompositions, perhaps the main theme of this paper, comprise a
subject completely alien to the classical theory of quadratic forms
over fields. Nothing similar to the quasilinear-rigid decomposition
seems to arise in the classical theory.
 (These decompositions will also be a major theme in
~\cite{QF2}.)

Since the rigid part is not unique, we describe
 the set of companions of a given quadratic form $q: V\to R$ in
 terms of the \textbf{companion matrix} (Theorem \ref{thm:I.3.3}), which is unique when the
 form $q$ is rigid. Companion matrices are obtained in the supertropical case
 by means of~
 Theorem~\ref{thm:I.3.13}.
In the special case that $R$ is a nontrivial tangible
supersemifield, the results in \S\ref{sec:I.4}, especially Theorems
\ref{thm:I.4.12} and \ref{thm:I.4.13}, allow a precise description
of all rigid complements of a given quadratic form, in
Theorems~\ref{thm:I.5.12} and \ref{thm:I.5.13}.

Finally, in \secref{sec:I.11}, we tie the theory to the classical
theory of quadratic forms,   spelling out the concept of
supertropicalization of a quadratic form by a supervaluation.
Supertropical semirings allow a refinement of valuation theory to a
theory of ``\emph{supervaluations}'', the basics of which can be
found in \cite{IzhakianKnebuschRowen2009Valuation}--\cite{IKR3}.
Supervaluations may provide an enriched version of tropical
geometry, cf. \cite[\S9, \S11]{IzhakianKnebuschRowen2009Valuation}
and \cite{IzhakianRowen2007SuperTropical}, as well as of tropical
matrix theory and the associated linear tropical algebra. They also
enable us to gain new insights for quadratic forms on a free module
$V$ over a ring $R$, in relation to an arbitrary base. This will be
indicated in \S\ref{sec:I.11}, where a strategy for further research
is laid out. But exploiting this concept in depth requires more
theory of quadratic forms and pairs over a supertropical semiring,
some of which is to be presented in \cite{QF2}.

The supertropicalization  process is up to now the main source of our interest in quadratic forms over semiring, a rather small class within the category of all semirings.  Supertropical quadratic forms are very rigid objects, well amenable to combinatorial arguments. One reason for this is \thmref{uniq} below (the Unique Base Theorem), which state that a free module over a supertropical semiring up to multiplication by units has only one base.

Given a quadratic form $q: V \to R $ on a free module $V$ over a ring $R$ and a base $\tL=(v_i\|i\in I)$ of $V$, further a supervaluation $\vrp:R \to U$ with values in a supertropical semiring $U$,  we obtain by $\vrp$ a quadratic form $\tlq : U^{(I)} \to U$ on the standard free $U$-module $U^{(I)}$. Our idea is to view the isometry class of $\tlq$ itself as an invariant of the pair $(q, \tL)$, which measures the base $\tL$ of $V$ by the quadratic form $q$ via the supervaluation $\vrp$. This invariant will be clumsy when  $I$ is big, but by structure theory of supertropical quadratic forms we can gain more manageable invariants from $\tlq$. (More on this will be said below in \S\ref{sec:I.11}.)

Concerning explicit computations, we emphasize that every base $\tL$ of $V$ is admitted here, while in classical quadratic form theory, even over a field $R$, often first a base of $V$ has to be established fitting the form and the problem (e.g., an orthogonal base if $\operatorname{char} R \neq 2$).

A very important fact here is that a supertropical semiring $U$ carries a natural partial ordering, called the \bfem{minimal ordering of} $U$, which is defined simply by the rule
 \begin{equation}\label{eq:0.7}
 x\le y \dss \Leftrightarrow \exists z\in U: x+z=y.\end{equation}
 A semiring with this property is called \textbf{upper bound} (abbreviated u.b.), a notion
 used in supertropical context already in \cite[\S11]{IKR3}. The minimal ordering allows arguments completely alien to classical quadratic form theory over rings. For convenience of the reader we will discuss upper bound semirings anew in \S\ref{sec:I.ub} below.

We advocate that in a full fledged supertropical quadratic form theory it is mandatory to admit also certain semirings, which are related to supertropical semirings, but themselves are not supertropical, e.g., polynomial function semirings in any number of variables over a supertropical semiring, cf. \cite[\S4]{IzhakianRowen2007SuperTropical}.
%A semiring $R$ is called \textbf{upper bound} (abbreviated u.b.) if the relation \eqref{eq:0.7} for $V = R$ is a partial ordering  on $R$, called again the \emph{minimal ordering on $R$}, cf. \cite[Definition 11.5]{IzhakianKnebuschRowen2009Valuation}. Polynomial function semirings are an example.
A modest theory of quadratic forms and quadratic pairs over an u.b. semiring seems to be a reasonable general frame in
which to place supertropical quadratic form theory.
Several  u.b.-results along these lines can be found in \secref{sec:I.3} and \secref{sec:I.5}.

\begin{notations}\label{notat:0.1} Let
$\N= \{1,2,3,\dots\},$ $\N_0=\N\cup\{0\}.$
If $R$ is a semiring, then $R^*$ denotes the group of units of $R$.

If $R$ is a supertropical semiring, then
\begin{itemize}
  \item
$\tT(R):= R\setminus eR=$ set of tangible elements $\ne0,$ \pSkip

  \item
$\tG(R): = eR\setminus\{0\}=$ set of ghost elements $\ne0,$
\pSkip

  \item
$\nu_R $ denotes the ghost map $R \to eR, $ $a\mapsto ea.$

\end{itemize}

When there is no ambiguity, we write $\tT$, $\tG$, $\nu$ instead of
$\tT(R),$ $\tG(R),$ $\nu_R.$

For $a\in R$ we also write $ea= \nu(a)=a^\nu.$

\end{notations}

\subsection*{Some aspects of supertropical algebra}\label{sec:I.1s}$ $

Let us set some terminology for this paper, and list some basic
supertropical results. An equivalent formulation, utilizing $e: =
1+1$ instead of the ghost map $\nu$, is as follows:

\begin{defn}[cf. {\cite[\S3]{IzhakianKnebuschRowen2009Valuation}}]\label{defn:0.2}
A semiring $R$ is \bfem{supertropical}, if $e$ is an idempotent
element (i.e., $1+1=1+1+1+1)$ and the following axioms hold for all
$x,y\in R.$
 \begin{equation}\label{eq:0.3} \quad
\text{If}\ ex\ne ey,\quad\text{then}\quad x+y\in\{x,y\}.
 \end{equation}
 \begin{equation}\label{eq:0.4}
\text{If}\ ex= ey,\quad\text{then}\quad x+y=ey.
 \end{equation}
% \begin{equation}\label{eq:0.5}
%\text{If}\ ex= 0,\quad\text{then}\quad x= 0.
% \end{equation}
\end{defn}

 If $R$ is supertropical, then $x+y\in \{x,y\}$ for any $x,y\in eR,$ and thus the ideal $eR$ is a bipotent semiring with unit element $e.$
 Any bipotent semiring $M$  carries a total ordering given by the rule
  \begin{equation}\label{eq:0.5}
a \leq b \dss \Leftrightarrow a + b= b.
 \end{equation}
 Using this ordering on $eR$, \eqref{eq:0.3} sharpens to the following condition $(x,y\in R):$
 \begin{equation}\label{eq:0.6}
\text{If}\ ex< ey,\quad\text{then}\quad x+y =  y.
 \end{equation}
% (cf. \cite[\S3]{IzhakianKnebuschRowen2009Valuation}).
 Conditions \eqref{eq:0.4} and \eqref{eq:0.6} show that addition on $R$ is determined by the ordering of the bipotent semiring $eR,$ the idempotent $e,$ and the map $\nu_R: R\to eR,$ $x\mapsto ex.$ $\nu_R$ is the ghost map, and $eR$ is the ghost ideal of $R$ mentioned above.

\begin{rem}\label{assump}  Any  supertropical semiring satisfies the following
conditions for all $a_, b\in R:$
\begin{enumerate} \eroman
\item $a+a=0\Rightarrow a=0,$
as follows from \eqref{eq:0.4}, applied to the elements $a$ and~$0.$
\pSkip \item $(a+b)^2=a^2+b^2.$ (See proof below.) \end{enumerate}
\end{rem}

 We prove Remark~\ref{assump} more generally for the following
important class of examples. Given a set $I$ and a semiring $U$, let $U^{I}$ denote the
semiring consisting of all $U$-valued functions on $I$, i.e., the
cartesian product $\prod_I U$ of copies of $U$ indexed by $I$.

\begin{prop}\label{prop:I.2.15}
If $R$ is any semiring of functions with values in a supertropical
semiring~$U$, i.e., $R$ is a subsemiring of $U^I$ for some set $I$,
then the  conditions of Remark~\ref{assump} hold.\footnote{Property
(ii) has been asserted before for $U$ supertropical, e.g. in
\cite[Proposition~ 3.9]{IzhakianRowen2007SuperTropical}, but the
proof there tacitly assumes that $U$ is cancellative, which holds
when $U$ is a supersemifield.}

\end{prop}
\begin{proof} Plainly it suffices to prove (ii) for $R =U.$ We may assume that $ea \leq eb.$ Then $e a^2 \leq eab
\leq eb^2.$ We distinguish three cases.
\begin{description}
  \item[Case 1] $ea < eb$, $ea^2 < eb^2.$ Now $a + b = b$, $a^2 + b^2 = b^2 = (a + b)^2.$ \pSkip

  \item[Case 2] $ea < eb$, $ea^2 = eb^2.$ Now $(a + b)^2 = b^2$. On the other hand
  $(a + b)^2 = a^2 + b^2+ eab =  eb^2+ eab.$ We infer that $b^2 = eb^2,$ whence
   $$a^2 + b^2 = eb^2 = b^2 = (a + b)^2.$$ \pSkip

  \item[Case 3]  $ea = eb$, $ea^2 = eb^2.$ Now $a + b = eb$, $a^2 + b^2 = eb^2 = (a + b)^2.$
 \end{description}
\end{proof}

\begin{defn}\label{defn:I.4.2} A supertropical semiring $R$ is   \bfem{tangible}, if $R$ is generated by~
$\tT(R)$ as a semiring. Clearly, this happens iff $e\tT(R)=\tG(R).$

 A supertropical semiring $R$ is a
\bfem{supersemifield}, if all tangible elements of~$R$ are
invertible in~$R$ and all ghost elements $\ne0$ of~$R$ are
invertible in the bipotent semiring~$e R.$
%\footnote{In
%\cite{IzhakianKnebuschRowen2009Valuation}, the tangible
%supersemifields have  been called ``supertropical semifields'' (loc.
%cit. \S\ref{sec:I.3}). We avoid this term here, since these
%semirings are not semifields in the technical sense, because ghost
%elements are not invertible.}
 \end{defn}

Then~$eR$ is a semifield, and $\tT(R)$ is the group of units of $R$,
except in the degenerate case  $\tT(R)=\emptyset,$ whereby
$R=eR.$\footnote{Discarding an uninteresting case, we also assume in
\S\ref{sec:I.4} that $R$ is \bfem{nontrivial}, i.e. $eR\ne\{0,e\}.$}

\begin{rem}\label{rem:I.4.3}
%\quad{} %\begin {enumerate}
%\item[a)]
If $R$ is a supertropical semiring  and $ \tT(R)\ne\emptyset,$ then
the set
$$R':=\tT(R)\cup e \tT(R)\cup\{0\}$$
is the largest subsemiring of $R$ which is tangible supertropical.
\pSkip
%
%\item[b)] A tangible supertropical semiring $R$ is a supersemifield
%iff every element of $\tT(R)$ is invertible in $R.$ Indeed, if
%$x\in \tT(R)$ is invertible in $R,$ then $ex$ is invertible in
%$eR$.

%\end{enumerate}

\end{rem}

\begin{exampl}[{\cite[\S3]{IzhakianKnebuschRowen2009Valuation}}]\label{stropi} Given a monoid $\tT$ and a totally ordered  cancellative monoid
$\tG$, together with a surjective homomorphism $\nu: \ptM \to \tG,$
the disjoint union $ \tT
\sqcup \tG$ is made into a monoid by starting with the given products on $\tT$ and $\tG$, and defining the products $ a
b^\nu $ and $ a^\nu b $ to be $ (ab)^\nu$ for $a,b \in \tT$. (We write $\nu(c) = c^\nu$ for $c \in \tT$.)

We extend $\nu$ to the \textbf{ghost map} $\nu:  \tT \sqcup \tG \to
\tG$ by taking $\nu|_\ptM = \nu$ and  $\nu_\tG$ to be the identity
on~$ \tG$. Thus, $\nu$ is a monoid projection.

We define addition on $ \tT \sqcup \tG$ by
 $$a+b =
\begin{cases} a \text{ for } a^\nu >  b^\nu;\\b \text{ for } a^\nu
< b^\nu;\\a^\nu \text{ for } a^\nu = b^\nu.
\end{cases}$$

We adjoin an element 0 to $\tT \sqcup \tG $ and put $0^\nu = 0$,
$0+x = x+0 = x,$ and $0\cdot x = x \cdot 0 = 0$ for all $x\in \tT
\sqcup \tG \sqcup \{0\}$. Then $R: =  \tT \sqcup \tG \sqcup \{0\}$
is a supertropical semiring, denoted by $\STR(\tT, \tG, \nu)$ and
called the \textbf{standard supertropical semiring}
arising from the monoid homomorphism $\nu : \tT \to \tG$.
\end{exampl}

When $\tG$ is a group, we call $R$  a \textbf{standard tangible
supersemifield}.
%
%Strictly speaking, only the tangible elements of $R$ are invertible,
%but we use the terminology ``tangible supersemifield'' since the
%theory runs parallel to linear algebras over semifields. This will
%be explained in greater algebraic precision in \S\ref{sec:I.11}.
%
%In particular, we are interested in $R = (R, \tT, \tG, \nu)$ where
%$\tT$ and $\tG$ are disjoint copies of~$T(\R)$, with $R = \tT \sqcup
%\tG \sqcup \{-\infty\}.$ Here   $0_R = -\infty$. %We define $\nu: R
%%\to \tG$ to be the projection whose restriction to $\tG$ is the
%%identity map and whose restriction to $\tT$ is the given isomorphism
%%$\nu: \tT \to \tG$.

\begin{thm}[Unique Base Theorem]\label{uniq} The  base     of a free module over a
supertropical semiring is
  unique  up to permutation of base elements and multiplication by
invertible elements of $R$.
\end{thm}
\begin{proof}The case of a semifield is given in \cite[Corollary~5.25]{IzhakianKnebuschRowen2010LinearAlg}, but the general
proof is not hard. Suppose $(\veps_i \ds |i\in I)$ and $(\veps'_j
\ds |j\in J)$ are bases, and write $\veps'_j = \sum _i a _{i,j}
\veps_i$ and $\veps_i = \sum _j b_{j,i} \veps'_j$ . Then for any $k
\in I,$
 $$\veps_{k} =   \sum _j \sum _i a _{i,j} b_{j,{k}} \veps_i,$$
 implying $\sum _j  a _{k,j} b_{j,{k}} = 1$ and $\sum _j b_{j,{k}} \sum _i a _{i,j} = 0$ for all $i \ne k.$
 Hence there is $\ell  = \pi(k)$ such that $ a _{k,\ell } b_{\ell ,{k}}   = 1$, implying  $a _{k,\ell }, b_{\ell ,{k}}  $ are invertible,
 and thus  $a _{i,\ell } = 0$ for all $i \ne k$, so $\veps'_{\ell } =   a _{k,\ell } \veps_{k}$ Likewise
  $\veps_{k} =   b _{\ell ,k} \veps'_{\ell }$. Doing this for each $k\in I$ yields the result.
\end{proof}

(Equivalently, the only invertible matrices over a supertropical
semiring are generalized permutation matrices.)

\section{Quadratic forms over a semiring}\label{sec:I.1}

We assume that $R$ is a semiring (always commutative) and $V$ is an
$R$-module. Recall Definition~\ref{defn:I.1.1}. Given a quadratic
pair $(q,b)$ on $V$, it follows from \eqref{eq:I.1.1} and
\eqref{eq:I.1.2} with $x=y$ that
\begin{equation}\label{eq:I.1.3}
4q(x)=2q(x)+b(x,x).\end{equation}

\begin{defn}\label{defn:I.1.2}
We call a quadratic pair $(q,b)$ on $V$ \bfem{balanced} if, for any
$x\in V$,
\begin{equation}\label{eq:I.1.4}
b(x,x)=2q(x).\end{equation}
\end{defn}

 \begin{remark}\label{rem:I.1.3}
If $R$ is a ring, then a quadratic pair  $(q,b)$ is determined by
 the quadratic form~ $q$ alone, namely
 \begin{equation}\label{eq:I.1.5} b(x,y)=q(x+y)-q(x)-q(y),\end{equation}
 and, moreover, $(q,b)$ is balanced.  If, in addition, $2$ is a unit in $R$, then
we have a bijection between  quadratic forms $q$ and symmetric
bilinear forms $b$      on $V$ via \eqref{eq:I.1.5} and
\begin{equation}\label{eq:I.1.6} q(x)=\frac{1}{2}b(x,x),\end{equation}
as is very well known.\end{remark}

But  our main interest is  in the case that $R$ is a supertropical
semiring (or even a  supersemifield). Then quadratic forms and
symmetric bilinear forms are only loosely related, and bilinear
forms which are not symmetric  play a major role. Also we will meet
many quadratic pairs which are not balanced.

\begin{examp}\label{examp:I.1.4}
Any bilinear form $B:V\times V\to R$ on $V$ (not necessarily
symmetric) gives us a balanced quadratic pair $(q,b)$ on $V$ as
follows $(x,y\in V)$:
\begin{equation}\label{eq:I.1.7} q(x):=B(x,x),\end{equation}
\begin{equation}\label{eq:I.1.8} b(x,y):=B(x,y)+B(y,x).\end{equation}
% Notice that $b$ is symmetric.
\end{examp}

\begin{example}\label{examp:I.1.6}
If $B: V\times V\to R$ is a {symmetric} bilinear form, then the
quadratic form $q(x):=B(x,x)$ together with the bilinear form
$b(x,y)=2B(x,y)$ is a balanced quadratic pair (cf. Example
\ref{examp:I.1.4}).
 \end{example}

We spell out what the equations \eqref{eq:I.1.3} and
\eqref{eq:I.1.4} mean if the semiring $R$ is bipotent or
supertropical.

\begin{remark}\label{rem:I.1.7} Let $(q,b)$ be a quadratic pair on $V$.
\begin{enumerate} \eroman
    \item Suppose that $R$ is  bipotent. Then $2 \cdot 1_R =
    1_R$, and thus \eqref{eq:I.1.3} reads $$q(x) = q(x) + b(x,x),$$ which
    means that\footnote{Recall that a bipotent semiring has a natural total ordering: $x \leq y$ iff $x+y = y.$ }
    \begin{equation}\label{eq:I.1.10}
b(x,x) \leq q(x).
\end{equation}
The pair $(q,b) $  is balanced iff for all $x \in V$
    \begin{equation}\label{eq:I.1.11}
b(x,x) =  q(x).
\end{equation}

    \item Assume that $R$ is supertropical. Now
$2 \cdot 1_R = e.$ Thus \eqref{eq:I.1.3} reads
$$ eq(x) = eq(x) + b(x,x)$$
which means that
        \begin{equation}\label{eq:I.1.12}
eb(x,x) \leq eq(x).
\end{equation}
%in the minimal ordering of $R$. \{Recall that a ghost element of
% $R$ is comparable with every element of $R$ in this ordering.\}
The pair $(q,b)$ is balanced iff
    \begin{equation}\label{eq:I.1.13}
b(x,x) = eq(x)
\end{equation}
for all $x\in V.$
\end{enumerate}

\end{remark}

\medskip
 We return to an arbitrary semiring $R$ and now focus on the case that $V$ is a free $R$-module with
 (classical) base $\veps_1,\dots,\veps_n$ and
 describe a quadratic pair $(q,b)$ in explicit terms. We write
 elements $x,y$ of $V$ as
 $$x=\sum_{i = 1}^n x_i\veps_i,\qquad y=\sum_{ i = 1} ^n
 y_i\veps_i , $$
 with $x_i,y_i\in\R.$ We use the abbreviations
 \begin{equation}\label{eq:I.1.14} \al_i:=q(\veps_i),\qquad \bt_{i,j}:=b(\veps_i,\veps_j).\end{equation}
Applying conditions \eqref{eq:I.1.1} and \eqref{eq:I.1.2} and
iterating, we obtain first
$$q(x)=q\bigg(\sum_{ i = 1} ^{n-1}x_i\veps_i\bigg)+\al_nx_n^2+\sum_{ i = 1}^{n-1}\bt_{i,n}x_ix_n,$$
and finally
\begin{equation}\label{eq:I.1.15}
q(x)=\sum_{ i = 1}^{n}\al_ix_i^2+\sum_{i< j} \bt_{i,j}
x_ix_j.\end{equation}

We furthermore have
 \begin{equation}\label{eq:I.1.16} b(x,y)=\sum_{i,j=1}^n
 \bt_{i,j} x_iy_j,\end{equation} with $\bt_{i,j}=\bt_{j,i},$ and, in consequence of \eqref{eq:I.1.3},
%
%$q(2\veps_i)=\al_{i}+\al_{i}+\bt_{i,i}.$
\begin{equation}\label{eq:I.1.17} 4\al_i=2\al_i+\bt_{i,i} .\end{equation}

\begin{defn}\label{defn:I.1.8}We call the $\al_i$ and the $\bt_{i,j}$ the \bfem{coefficients} of
the quadratic pair $(q,b) $ with respect to the base
$\veps_1,\dots,\veps_n.$\end{defn}

\begin{prop}\label{prop:I.1.9}
The quadratic pair $(q,b)$ is balanced iff $\bt_{i,i}=2\al_i$ $(1\le
i\le n).$\end{prop}

\begin{proof} The equations $\bt_{i,i}=2\al_i$ are special cases of the
balancing rule \eqref{eq:I.1.4}. Assume that these equations hold.
Then indeed, for any $x=\sum_i x_i\veps_i,$
\begin{align*}
b(x,x)&=\sum_{i,j} \bt_{i,j}x_ix_j \\ & =\sum_i \bt_{i,i}
x_i^2+2\sum_{i<j}
\bt_{i,j}x_ix_j\\
&=2\bigg[\sum_i \al_ix_i^2+\sum_{i<j}\bt_{i,j}x_ix_j\bigg]
=2q(x).\end{align*}\end{proof}

If $R$ is not a ring, then   different polynomials in $x_1,\dots,
x_n$ often give the same function on $V.$ In order not to complicate
our setting too much at present, we now redefine (in the case of a
free module $V)$ a quadratic form as a polynomial on $V,$ a notion
which with sufficient care can be regarded as independent of the
choice of the base $\veps_1,\dots,\veps_n.$ Accordingly, we also
view a bilinear form on $V$ as a polynomial in variables $x_1,\dots,
x_n,$ $y_1,\dots,y_n.$

Whenever necessary, we specify the quadratic forms and pairs in
this polynomial sense \bfem{formal quadratic forms} and pairs,
while for the forms and pairs in Definition \ref{defn:I.1.1} we use
the term ``\bfem{functional}''. The dependence of a formal pair
$(q,b)$ on a given base of the free $R$-module is not a
conceptual problem, since it is obvious how to rewrite $q$ and $b$
with respect to another base.

We now work with formal quadratic forms and pairs, keeping fixed
the base $\veps_1,\dots,\veps_n$ of~ $V$. In the following we
often denote a bilinear form $B$ on $V$ by its ``Gram-matrix"
\begin{equation}\label{eq:I.1.19}
B=\begin{pmatrix} \bt_{1,1} & \dots & \bt_{1,n}\\
\vdots & \ddots & \vdots \\   \bt_{n,1}& \dots &
\bt_{n,n}\end{pmatrix},
\end{equation}
$\bt_{i,j} :=B(\veps_i,\veps_j),$ and we denote by $\trn{B}$ the
bilinear form given by the transpose of the matrix~ $B$. On the
functional level $\trn{B}(x,y) = B(y,x).$

\begin{defn}\label{defn:I.1.10} We call a bilinear form $B$ on $V$
\bfem{triangular}, if all coefficients $\bt_{i,j}$ with $i>j$ are
zero; hence
\begin{equation}\label{eq:I.1.20}
B=\begin{pmatrix} \bt_{1,1} & \bt_{1,2}& \dots & \bt_{1,n}\\
& \bt_{2,2} & & \vdots\\
  & &\ddots & \bt_{n-1,n}\\
0& & & \bt_{n,n}\end{pmatrix} . \end{equation}\end{defn}

This property strongly depends on the choice of the base
$\veps_1,\dots,\veps_n$ of $V.$

\begin{defn}\label{defn:I.1.11} Given a quadratic form $q$ and a
bilinear form $B$ on $V$, we say that $q$ \bfem{admits}~$B,$ or that
$B$ \bfem{expands} $q$, if $B(x,x)=q(x)$ (cf. Example
\ref{examp:I.1.4}). In explicit terms this means that, if
$B=(\bt_{i,j}),$ then
\begin{equation}\label{eq:I.1.21} q(x)=\sum_{i,j=1}^n
\bt_{i,j}x_ix_j.\end{equation}
\end{defn}

The following is now obvious.

\begin{prop}\label{prop:I.1.12}
 A given quadratic form $q$ on $V$ with coefficients $\al_i$ $(1\le i\le
n),$ and $\bt_{i,j}$ $(1\le i,j\le n),$ (cf. Definition
\ref{defn:I.1.8}) expands to a unique triangular bilinear form,
denoted by~$\triangledown q,$ namely,
 \begin{equation}\label{eq:I.1.22}
 \triangledown q=\begin{pmatrix}  \al_{1} & \bt_{1,2}& \dots & \bt_{1,n}\\
& \al_{2} & &  \vdots \\
  & &  \ddots & \bt_{n-1,n}\\
0 &  & & \al_{n}\end{pmatrix}. \end{equation} The symmetric bilinear
form
\begin{equation}\label{eq:I.1.23}
  b_q =\begin{pmatrix}  2\al_{1} & \bt_{1,2}& \dots & \bt_{1,n}\\
\bt_{1,2} & 2\al_{2} & &   \vdots \\
  \vdots & &  \ddots &   \bt_{n-1,n}   \\
\bt_{1,n} &  \cdots & \bt_{n,n-1}& 2\al_{n}\end{pmatrix}\end{equation}
is the unique one which completes the formal quadratic form $q$ to
a balanced quadratic pair~ $(q,b_q)$.
\end{prop}

\begin{notations} \label{notat:I.1.13} $ $
 We denote the quadratic form
$q$ with coefficients $\al_i$ $(1\le i\le n)$ and $\bt_{i,j}$ $(1\le
i,j\le n)$ by the triangular scheme\footnote{Such
triangular schemes have already been used in the literature in the
case that $R$ is a ring, cf., e.g., \cite[I,\S2]{Kneser}.}
\begin{equation}\label{eq:I.1.24}
q=\begin{bmatrix} \al_1& \bt_{1,2} &\dots& \bt_{1,n}\\
 &\al_2 & \ddots & \vdots  \\
 &   & \ddots & \bt_{n-1, n} \\
 & &  &   \al_n\end{bmatrix},\end{equation}
showing all coefficients of the homogeneous polynomial $q(x)$. It
is the matrix of $\triangledown q$ in square brackets.
We also say that \eqref{eq:I.1.24} (or \eqref{eq:I.1.15}) is a presentation of the functional quadratic form $q$.

 \end{notations}

 The notation of \eqref{eq:I.1.24} describes a
(formal) quadratic form $q$ on $V$ by parameters $\al_i$, $\bt_{i,j}$,
which are uniquely determined by $q$ and the given base of $V.$
%Sometimes it will be convenient to use a more flexible notation
%abandoning uniqueness.
%
%
%\begin{notation}\label{notat:I.1.14}
%Let $A$ be any $n \times n $ matrix with coefficients in $R$.  We
%interpret  $A$ as the matrix of a bilinear form $B$ on $V$ and
%write
%\begin{equation}\label{eq:I.1.26} q=[A]\end{equation}
%for the quadratic form $q(x) := B(x,x)$ which expands to $B$.
%\end{notation}

We return to functional quadratic forms and pairs.
\begin{defn}\label{defn:I.1.15} \  Let $q$ be a functional quadratic form
on the $R$-module $V.$ We call a symmetric bilinear form $b$ on
$V$ a \textbf{companion} of $q$ (or say that $b$
\textbf{accompanies} $q$), if~ $b$ completes~ $q$ to a functional
quadratic pair $(q,b)$. When the pair $(q,b)$ is balanced, we call
$b$ a \textbf{balanced companion} of $q.$
\end{defn}

As a consequence of Proposition \ref{prop:I.1.12} we state

\begin{cor}\label{cor:I.1.16} Every functional quadratic form $q$ on
a free $R$-module has a balanced companion $b$ (but perhaps not
unique, cf.~\secref{sec:I.2}).
\end{cor}

Our main interest in this paper is in functional quadratic forms and
pairs, while formal forms and pairs will usually serve to denote
functional forms and pairs in a precise and efficient way (in the
case that the $R$-module $V$ is free). Most often it should be clear
from the context whether the forms and the pairs at hand are formal
or functional.

We fix a companion $b_0:V\times V\to R$ of a functional quadratic
form $q$ (which exists by Definition \ref{defn:I.1.1}), and now
search for other companions of $q$ in a systematic way. Here the
notion of a ``\textit{partial companion}" of~$q$ will turn out to be
helpful.

\begin{defn}\label{defn:I.2.1}
Let $b: V\times V\to R$ be a symmetric bilinear form. We say that
$b$ is a \bfem{companion of} $q$ \bfem{at} a point $(x,y)\in V\times
V$ (or, that $b$ \bfem{accompanies} $q$ \bfem{at} $(x,y)$), if
\begin{equation}\label{eq:I.2.1}
q(x+y)=q(x)+q(y)+b(x,y),\end{equation} which can be rephrased as
\begin{equation}\label{eq:I.2.2}
q(x)+q(y)+b(x,y)=q(x)+q(y)+b_0(x,y).\end{equation} If this happens
to be true for every point $(x,y)$ of a set $T\subset V\times V,$ we
call $b$ a \bfem{companion of} $q$ \bfem{on} $T;$ and in the subcase
$T=S\times S,$ where $S\subset V,$ we say more briefly that $b$ is a
\bfem{companion of} $q$ \bfem{on} $S.$ We then also say that $b$
\bfem{accompanies} $q$ on $T,$ resp. on $S.$\end{defn}

\begin{remark}\label{rem:I.2.2}
The  bilinear form $b$ accompanies $q$ on an $R$-submodule $W$ of
$V$ iff $b|W\times W$ is a companion of $q|W$.  \end{remark}

\begin{examp}\label{examp:I.2.4}
Under   condition (ii) of Remark~\ref{assump}, let $x_0\in V.$ A
symmetric bilinear form~ $b$ on~ $V$ accompanies $q:V\to R$ on
$Rx_0$ iff for all $\lm,\mu\in R\sm \{0\}:$
\begin{equation}\label{eq:I.2.3}
(\lm^2+\mu^2)q(x_0)=(\lm^2+\mu^2)q(x_0)+\lm\mu
b(x_0,x_0).\end{equation}
% due to the fact that
%$\lm^2+\mu^2=(\lm+\mu)^2.$
\end{examp}

\begin{examp}\label{examp:I.2.5}
Let $x_0,y_0\in V.$ A symmetric bilinear form $b$ accompanies $q:
V\to R$ on $Rx_0\times Ry_0$ iff for all $\lm,\mu\in R\sm\{0\}$:
\begin{equation}\label{eq:I.2.4} \lm^2q(x_0)+\mu^2q(y_0)+\lm \mu
b_0(x_0,y_0)=\lm^2q(x_0)+\mu^2q(y_0)+\lm\mu
b(x_0,y_0).\end{equation} (Here $R$ may be any semiring.)
\end{examp}

In the following, $q$ always denotes a quadratic form on an
$R$-module $V$ and $b$ a symmetric bilinear form on $V.$ Given a set
$T\subset V\times V$ on which $b$ accompanies $q$, we look for ways
to obtain a bigger set $T'\supset T$ on which $b$ accompanies $q.$
As above $b_0$ denotes a fixed companion of $q$.

\begin{lemma}\label{lem:I.2.6}
If $b$ accompanies $q$ at the points $(x,y),(x',y),$ and $(x,x'),$
then $b$ accompanies ~$q$ at the point $(x+x',y).$\end{lemma}

\begin{proof}
We will use the equations \begin{align}  &
 q(x)+q(y)+b(x,y)=q(x)+q(y)+b_0(x,y), \tag{$\al$}\\
&  q(x')+q(y)+b(x',y)=q(x')+q(y)+b_0(x',y),\tag{$\bt$}\\
& q(x)+q(x')+b(x,x')=q(x)+q(x')+b_0(x,x'),\tag{$\gm$}\\
\intertext{indicated by \eqref{eq:I.2.2}, and furthermore}
 &q(x+x') =q(x)+q(x')+b(x,x'),\tag{$\delta$}\\
& q(x+x') =q(x)+q(x')+b_0(x,x')\tag{$\veps$}.\end{align}  Applying
$(\delta)$ we obtain
$$q(x+x')+q(y)+b(x+x',y)=q(x)+q(x')+q(y)+b(x,x')+b(x,y)+b(x',y).$$
Applying $(\al)$, $(\bt)$, and $(\gm)$ to the right hand side, we
replace $b$ by $b_0$ throughout. Finally, applying $(\veps)$ we
obtain
$$q(x+x')+q(y)+b(x+x',y)=q(x+x')+q(y)+b_0(x+x',y),$$
as desired.
\end{proof}

\begin{lemma}\label{lem:I.2.7}
Assume that $b$ accompanies $q$ on a set $S\subset V.$ Then $b$
accompanies $q$ on the submonoid $\langle S\rangle$ of $(V,+)$
generated by $S.$
\end{lemma}

\begin{proof}
We have
$$\langle S\rangle=\{x_1+\dots+x_r \ds | r\in\N,\, x_i\in
S\}\cup\{0\}.$$ It is trivial that~ $b$ accompanies~ $q$ on $V\times
\{0\}.$ We prove by induction on $r$ that~ $b$ accompanies~ $q$ at
every point $(x_1+\dots+x_r,y)$ with $x_i, y\in S.$ This is true by
assumption for $r=1.$ Assuming the claim for sums $x_1+\dots +x_s$
with $s<r,$ we see that $b$ accompanies $q$ at $(x_1+\dots
+x_{r-1},x_r).$ We conclude by \lemref{lem:I.2.6} that $b$
accompanies $q$ at $(x_1+\dots +x_r,y).$ Thus $b$ accompanies $q$ on
$\langle S\rangle\times S.$ By an analogous induction we conclude
that $b$ accompanies $q$ on $\langle S\rangle\times\langle
S\rangle.$
\end{proof}

\begin{prop}\label{prop:I.2.8}
Assume that $(\veps_i\ds|i\in I)$ is a subset of $V$ which generates
$V$ as an $R$-module, and assume that $b$ accompanies $q$ on the set
$S=\bigcup_{i\in I}R\veps_i.$ Then $b$ is a companion of
$q.$\end{prop}

\begin{proof}
Now $\langle S\rangle=V.$ \lemref{lem:I.2.7} applies.\end{proof}

\section{Quasilinear quadratic forms}\label{quasilin}

Companionship of the null bilinear form deserves special attention.
This is reflected in the following terminology.

\begin{defn}\label{defn:I.1.5}
A  quadratic form $q: V\to R$ is called
\bfem{quasilinear}\footnote{In the case that $R$ is a valuation
domain, cf. \cite[I,\S6]{Spez}. In \cite{EKM} these forms are called
``totally singular".} if $q$, together with the null form $b: V
\times V\to R,$ $b(x,y) =0$ for all $x,y\in V, $ is a quadratic
pair. This means that for any $x,y\in V,$ $a\in  R$,
\begin{equation}\label{eq:I.1.9} q(ax)=a^2q(x),\qquad q(x+y)=q(x)+q(y).\end{equation}
\end{defn}

Clearly a form $q$ on $V$ is quasilinear iff
\begin{equation}\label{eq:I.5.6}
q\Big(\sum_i a_i\veps_i\Big)=\sum_i a_i^2 q(\veps_i).
\end{equation}

\begin{example}
 If in Example~\ref{examp:I.1.6}  $2=0$ in $R,$ then the form $q$
is quasilinear.\end{example}

Quasilinear forms over a supertropical semifield have been
considered in \cite[\S5]{IKR-LinAlg2}. It is useful to refine the
definition.

\begin{defn}\label{defn:I.2.3}
A quadratic form $q$ is \bfem{quasilinear on} a set $T\subset
V\times V$, if $b=0$ accompanies~ $q$ on~ $T,$ i.e., for any
$(x,y)\in T$,
$$q(x+y)=q(x)+q(y);$$
and we say that $q$ is quasilinear on $S\subset V$ if this happens
on $T=S\times S.$\end{defn}

\begin{examp}\label{examp:I.2.4a}
If $4 \cdot 1_R=2 \cdot 1_R,$ for example, if $R$ is supertropical,
then every quadratic form on $V$ is quasilinear on the diagonal of
$V\times V.$\end{examp}

If $q$ is quasilinear and $V$ is free with base $(\veps_i \ds |i\in I)$, then Equation~\eqref{eq:I.1.15} becomes:
\begin{equation}\label{eq:I.1.18} q(x)=\sum_{i=1}^n \al_ix_i^2 ,\end{equation}

or, as a triangular scheme (Notation~\ref{notat:I.1.13}):
$$q=\begin{bmatrix} \al_1 & 0 &\dots &
0\\
& \al_2 & \ddots & \vdots \\
& & \ddots & 0\\
& & & \al_n\end{bmatrix} . $$ For this form, we usually switch to
simpler notation with one row:
\begin{equation}\label{eq:I.1.25} q=[\al_1,\al_2,\dots, \al_n].\end{equation}

\section{Rigidity}\label{sec:I.2}

In this section we address the question of uniqueness of a partial
companion of $q.$

\begin{defn}\label{defn:I.2.10}
We say that $q$ is \bfem{rigid at a point} $(x,y)$ of $V\times V,$
if $b_1(x,y)=b_2(x,y)$ for any two companions $b_1,b_2$ of $q$ at
$(x,y).$ We rephrase this as
\begin{equation}\label{eq:I.2.5}
b(x,y)=b_0(x,y)\end{equation} for any companion $b$ of $q$ at
$(x,y),$ where, as before, $b_0$ denotes a given companion of~ $q$
(on ~$V).$ If $q$ is rigid at every point $(x,y)$ of a set
$T\subset V\times V$, we say that $q$ is \bfem{rigid on} $T.$
Finally, if $q$ is rigid on $S\times S$ for some $S\subset V,$ we
more briefly say that $q$ is \bfem{rigid on} $S,$ and in the
subcase $S=V$ we say that $q$ is a \bfem{rigid quadratic form}.
\end{defn}

Later we will make use of the following easy fact.

\begin{prop}\label{prop:I.2.10}
Assume that $S_1,S_2$ are subsets of $V,$ and suppose that $q$ is
rigid on $S_1\times S_2.$ Let $W_i$ denote the $R$-submodule of
$V$ generated by $S_i$ $(i=1,2).$  Then $q$ is rigid on~ $W_1\times
W_2.$\end{prop}

\begin{proof} Let $b$ be a companion of $q$ on $S_1\times S_2.$
Then $b|S_1\times S_2=b_0|S_1\times S_2.$ From this one concludes
easily by bilinearity that $b|W_1\times W_2=b_0|W_1\times
W_2.$\end{proof}

\begin{cor}\label{cor:I.2.11} Assume that $(\veps_i \ds |i\in I)$ generates
 the $R$-module $V,$ and suppose that $q$ is
rigid at $(\veps_i,\veps_j)$ for any $i,j\in I.$ Then $q$ is a
rigid quadratic form.\end{cor}

Remarkably, the following rigidity result holds over any semiring
$R.$

\begin{prop}\label{prop:I.2.12}
Any quadratic form $q$ on an $R$-module $V$ is rigid on the
submodule $W$ of~$V$ generated by the set $\{x\in V \ds
|q(x)=0\}.$\end{prop}

\begin{proof} Let $S:=\{x\in V \ds |q(x)=0\}.$ Given two companions
$b_1$ and $b_2$ of $q$ on $S$, we have
$$b_1(x,y)=q(x+y)=b_2(x,y)$$
for every $(x,y)\in S\times S.$ Thus $q$ is rigid on $S.$ By
\propref{prop:I.2.10} it follows, that $q$ is rigid on
$W.$\end{proof}

Now we obtain an explicit criterion for a
 quadratic form on~$V$ to be rigid.

\begin{thm}\label{thm:I.2.13} Assume that $V$ is free with base
$(\veps_i \ds|i\in I)$, and that  $R$ satisfies the conditions of
Remark~\ref{assump}.
 Then a quadratic form
$q$ on $V$ is rigid iff $q(\veps_i)=0$ for every $i\in I.$
\end{thm}

\begin{proof} If all $q(\veps_i)=0,$ we conclude by
\propref{prop:I.2.12} that~$q$ is rigid. Now assume that~$q$ is
rigid, but that there exists some $k\in I$ with $q(\veps_k)\ne0.$
Then the balanced companion~$b_q$ described in \propref{prop:I.1.12}
(after choosing a formal representative of $q,$ again denoted by~
$q)$ has value $b_q(\veps_k,\veps_k)=2q(\veps_k)\ne0,$  by Condition
(i) of Remark~\ref{assump}. On the other hand, the null form $b=0$
accompanies~$q$ at $(\veps_k,\veps_k),$ as follows from Condition
(ii)   (cf.~Example~ \ref{examp:I.2.4}). Since $q$ is rigid at
$(\veps_k,\veps_k)$, and both $b_q$ and $b = 0$ accompany $q$ at
$(\veps_k,\veps_k)$, this is a contradiction.
\end{proof}

\section{The quasilinear-rigid decomposition}\label{sec:I.7}

We assume throughout this section that $V$ \textit{is a free module} with base $(\veps_i \ds |i\in I)$
over a semiring $R$ satisfying the conditions of
Remark~\ref{assump}.

\begin{prop}\label{prop:I.5.1}
Let $b$ be a companion of $q: V\to R$.   We define quadratic forms
$\ql$ and $\rg$ on $V$ by the formulas $(x=\sum_{i\in I}
x_i\veps_i\in V)$
\begin{equation}\label{eq:I.5.3}
\ql(x)=\sum_{i\in I}q(\veps_i)x_i^2\end{equation}
\begin{equation}\label{eq:I.5.4}
\rg(x)=\sum_{i<j}b(\veps_i,\veps_j)x_ix_j,\end{equation} where we
have chosen a total ordering on the set of indices $I.$ Then $\ql$
is quasilinear, $\rg$ is rigid, and $q=\ql+\rg.$
\end{prop}

\begin{proof}
It is obvious that $q=\ql +\rg$ (cf. \eqref{eq:I.1.15}).  Clearly
$\rg(\veps_i)=0$ for all $i\in I,$ whence $\rg$ is rigid. As a
consequence of Remark~\ref{assump}(ii), we have
$$\ql(x+y)=\ql(x)+\ql(y)\quad \text{for all}\quad x,y\in V,$$
whence $\ql$ is quasilinear.
\end{proof}

\begin{prop}\label{prop:I.5.2}
If $q=\ql+\rg$ with $\ql$ quasilinear and $\rg$ rigid, then $\ql$
satisfies the formula~ \eqref{eq:I.5.3} for all $x\in V,$ and hence
$\ql$ is uniquely determined by $q.$ In particular, $\ql$ does not
depend on the choice of the base $(\veps_i).$
\end{prop}

\begin{proof}
For every $i\in I$ we have $\rg(\veps_i)=0,$ and hence
$q(\veps_i)=\ql(\veps_i).$ Since $\ql$ is quasilinear, this implies
\eqref{eq:I.5.3}.
\end{proof}

\begin{defn}\label{defn:I.5.3}
We call $q_{\QL}+\rg=q$ (with $q_{\QL}$ quasilinear  and $\rg$
rigid) a \bfem{quasilinear-rigid decomposition} of $q$.  Furthermore
we call  $q_{\QL}$ the \bfem{quasilinear part} of $q.$  Furthermore
we call $\rg$  a \bfem{rigid complement of} $q_{\QL}$ \bfem{in} $q,$
or more briefly, a \bfem{rigid complement in} $q,$ justified by the
fact that $q_{\QL}$ is uniquely determined by $q.$ We denote the set
of all rigid complements in $q$ by $\Rig(q).$
\end{defn}

Rigid complements in $q$ are closely related to certain companions
of $q.$

\begin{defn}\label{defn:I.5.4}
We call a companion $b$ of $q$ \bfem{off-diagonal} if
$b(\veps_i,\veps_i)=0$ for every $i\in I.$\end{defn}

\begin{remark}\label{rem:I.5.5}
We have $0\in C_{ii}(q)$ for every $i\in I,$ since $q$ is
quasilinear on $R\veps_i$ due to Condition (ii) of
Remark~\ref{assump}. Thus every form $q$ on $V$ possesses
off-diagonal companions.
\end{remark}

\begin{prop}\label{prop:I.5.6}
The rigid complements $\rg$ in $q$ correspond bijectively with the
off-diagonal companions $b$ of $q$ by the formulas \eqref{eq:I.5.4}
and
\begin{equation}\label{eq:I.5.5}
b(\veps_i,\veps_j)=\rg(\veps_i+\veps_j)
\end{equation}
for $i,j\in I.$ The bilinear form $b$ is the (unique) companion of
$\rg.$
\end{prop}

\begin{proof}
Given a companion $b$ of $q$ we obtain a rigid complement $\rg$ in
$q$ by formula \eqref{eq:I.5.4}, as observed in
\propref{prop:I.5.1}, and \eqref{eq:I.5.4} implies \eqref{eq:I.5.5}.
If moreover $b$ is off-diagonal, then $b$ is uniquely determined by
\eqref{eq:I.5.5}.

Conversely, if $q=\ql+\rg$ with $\ql$ quasilinear and $\rg$ rigid,
then $\ql$ has the companion $b_0=0$, while $\rg$ has a (unique)
companion $b.$ Thus $b_0+b=b$ is a companion of $\ql+\rg=q.$ We have
$$4\rg(\veps_i)=\rg(\veps_i+\veps_i)=\rg(\veps_i)+\rg(\veps_i)+b(\veps_i,\veps_i),$$
and conclude that $b(\veps_i,\veps_i)=0,$ since $\rg(\veps_i)=0.$
Thus $b$ is off-diagonal.\end{proof}

\begin{schol}\label{schol:I.5.7}
 A form $\rg$ on $V$ is rigid iff
\begin{equation}\label{eq:I.5.7}
\rg\Big(\sum_ix_i\veps_i\Big)=\sum_{i<j}\rg(\veps_i+\veps_j)x_ix_j.\end{equation}
\end{schol}

\begin{notation}\label{notat:I.5.8}
We denote the set of all quadratic forms on $V$ by $\Quad(V)$, and
view this set as an $R$-module of $R$-valued functions in the
obvious way:
$$
\begin{array}{rl}
  (q_1+q_2)(x)  := & q_1(x)+q_2(x), \\[2mm]
  (\lm q)(x) := & \lm\cdot q(x).
\end{array}
$$
for every $x\in V$ $(q_1,q_2,q\in\Quad( V),$ $\lm\in R).$ We further
denote the set of quasilinear forms on $V$ by $\QL(V)$ and the set
of rigid forms on $V$ by $\Rig(V).$
\end{notation}

\begin{remarks}\label{rem:I.5.9} $ $
\begin{enumerate}

\item[a)] It is evident that both $\QL(V)$ and $\Rig(V)$ are submodules of the $R$-module $\Quad(V).$ \pSkip

\item[b)] As a consequence of Propositions \ref{prop:I.5.1}, \ref{prop:I.5.2}, we have
\begin{equation}\label{eq:I.5.8}
    \QL(V)+\Rig(V)=\Quad(V),\end{equation}
\begin{equation}\label{eq:I.5.9}
    \QL(V)\cap\Rig(V)=\{0\}.\end{equation}
\pSkip

\item [c)] For any $q_1,q_2\in\Quad(V)$
\begin{equation}\label{eq:I.5.10}
(q_1+q_2)_{\QL}=(q_1)_{\QL}+(q_2)_{\QL},
\end{equation}
and
\begin{equation}\label{eq:I.5.11}
\Rig(q_1)+\Rig(q_2)\subset\Rig(q_1+q_2).
\end{equation}
\pSkip

\item[d)] For any $q\in\Quad(V)$ and $\lm\in R$
\begin{equation}\label{eq:I.5.12}
(\lm q)_{\QL}=\lm\cdot q_{\QL},
\end{equation}
and
\begin{equation}\label{eq:I.5.13}
\lm\cdot\Rig(q) \subset\Rig(\lm q).
\end{equation}

Here the assertions about rigid complements in c), d) are obvious
from Definition~\ref{defn:I.5.3}, while the assertions about
quasilinear parts follow from the formula \eqref{eq:I.5.3}.
% describing $\ql=q_{\QL}.$
\end{enumerate}

\end{remarks}

\section{Upper bound semirings}\label{sec:I.ub}

On any semiring $R$ we have a binary relation $\leq_R$ as follows ($x,y \in R$):
\begin{equation}\label{eq:I.ub.1} x\leq_R y\quad \Leftrightarrow\quad \exists
z\in R:x+z=y.\end{equation}
Clearly this relation is reflexive ($x \leq_R x$) and transitive ($x \leq_R y, y \leq_R z \Rightarrow x \leq_R z$).

\begin{defn}\label{defn:I.ub.1} \footnote{The notion of an u.b. semiring appeared in a supertropical context already in \cite[\S11]{IzhakianKnebuschRowen2009Valuation}, but under another perspective, cf. \cite[Definition 11.5]{IzhakianKnebuschRowen2009Valuation}. We thank the referee for suggesting a more elaborate discussion of minimal ordering here.}
A semiring $R$ is called
\textbf{upper bound}  (abbreviated  u.b.) if $\leq_R$ is a partial order on the set $R$, i.e.,
$x \leq_R y, y \leq_R x \Rightarrow x = y$; this ordering is then called the \textbf{minimal ordering} on~$R$.
\end{defn}

The term ``upper bound'' refers to the property that the sum $x+y$ is a built in upper bound of the set $\{ x,y\} $ in $R$ with respect to $\le_R$ (which perhaps is not a minimal upper bound).

\begin{rem}\label{rem:I.ub.2}
When $R$ is u.b., the set $R \sm \{ 0\}$ is closed  under addition. Indeed, if $x+y = 0 $ then
$x \leq_R 0$. Since $0 + x = x$, also $0 \leq_R x$. Thus $x = 0$ and $y= 0$.
\end{rem}

\begin{examp}\label{exmp:I.ub.3}
Every bipotent semiring $R$ is upper bound, and the minimal ordering on $R$ coincides  with the total ordering on $R$ mentioned in the introduction, cf. \eqref{eq:0.5}, i.e.,
\begin{equation}\label{eq:I.ub.2} x\leq_R y\quad \Leftrightarrow\quad x+y=y.\end{equation}
\end{examp}
Indeed, if $x+ z = y$ for some $z \in R$, then $x+y = x+x+z = x+z =y$, because $x+x = x$. The other implication is trivial. Note that now $x+y$ is always the maximum of $\{ x,y\}.$

\begin{defn}[{cf. \cite[Chap. 20]{golan92}}]\label{defn:I.ub.4}
When $R$ is a (commutative) semiring then a partial ordering $\al$ on the set $R$ is called a \textbf{positive partial semiring ordering}, if
\begin{equation}\label{eq:I.ub.3} \qquad x\leq_\al y\quad \Rightarrow\quad x+z \leq_\al y +z,\end{equation}
\begin{equation}\label{eq:I.ub.4} x\leq_\al y\quad \Rightarrow\quad xz \leq_\al yz,\end{equation}
for all $x,y,z \in R$, and furthermore
\begin{equation}\label{eq:I.ub.5} 0 \leq_\al 1.\end{equation}
\end{defn}
\noindent \{Thus $0 \leq_\al x$ for all $x\in R$, and $x\leq_\al y$, $ z\leq_\al w$ implies that $x+z \leq_\al y + w,\ xz\leq_\al yw$.\}\footnote{These partial orderings retain sense when  $R$ is not commutative, by modifying \eqref{eq:I.ub.4} in an obvious way.}
\pSkip

The term ``minimal ordering'' alludes to the following facts.

\begin{prop}\label{prop:I.ub.5} Let $R$ be a semiring.
\begin{enumerate}
  \item[a)] The relation $\le_R$ (cf. \eqref{eq:I.ub.1}) obeys the rules \eqref{eq:I.ub.3}--\eqref{eq:I.ub.5}, whence is a positive partial semiring ordering if $R$ is u.b. \pSkip
  \item[b)] If $R $ carries a positive partial semiring ordering $\al$, then $R$ is u.b. and $\al$ is a refinement of the ordering $\leq_R$, i.e.,
\end{enumerate}
\begin{equation*} x\leq_R  y\quad \Rightarrow\quad x \leq_\al y.\end{equation*}

\end{prop}
\begin{proof} a): The rules \eqref{eq:I.ub.3}--\eqref{eq:I.ub.5} are straightforward consequence of \eqref{eq:I.ub.1}. \pSkip

b): Let $x \leq_R y$. We have $y = x+z$ for some $z \in R $. By definition $0 \leq_\al z $ and then
$x \leq_\al x +z = y.$
\end{proof}

\begin{cor}\label{cor:I.ub.7}
Any subsemiring $A$ of an u.b. semiring $R$ is again u.b.
\end{cor}
\begin{proof} The relation $\leq_R$ is a positive partial semiring ordering on $R$, whence its restriction to $A$ is a positive partial semiring ordering on $A$. \propref{prop:I.ub.5}.b gives the claim.
\end{proof}

Of course it may happen that the restriction of $\leq_R$ to $A$ is strictly finer than $\leq_A$. \pSkip
Our interest here in u.b. semirings arises from the following

\begin{prop}\label{prop:I.ub.8} Let $R$ be a supertropical semiring and $M := eR$. Then
\begin{enumerate}
  \item[a)] $R$ is u.b.; \pSkip
  \item[b)] the restriction of $\leq_R$ to the set $M$ is the minimal ordering $\leq_M$ of the bipotent semiring~$M$.
\end{enumerate}

\end{prop}
\begin{proof} a): Assume that $x,y,z,w$ are elements of $R$ with $x+z = y$, $y+w = x$. Then
$x + (z+w) = x$, and  so $$x + e(z+w) = x + (z+w) = x.$$
Adding  $z$ on both sides and observing that $ez + z = ez$, we obtain
$$ y = x+z =  x + e(z +w) + z = x + e(z +w ) = x,$$
as desired.
\pSkip
b): Let $ex \leq_R ey$. There exists $z \in R$ with $ex +z = ey$. Multiplying by $e$ we obtain
$ex + ez = ey$, whence $ex \leq_M ey.$
\end{proof}

\begin{rem}\label{rem:I.ub.9} If $(R_i \ds | i \in I)$ is a family of semirings,  let
$R := \prod_{i\in I} R_i$. Given elements $x = (x_i \ds | i \in I)$, and $y = (y_i \ds | i \in I)$ of $R$, it is obvious that
\begin{equation*} x\leq_R  y\quad \Rightarrow\quad \forall i\in I: \ x_i \leq_{R_i} y_i.\end{equation*}
Thus $R$ is u.b. iff $R_i$ is u.b. for every $i \in I.$
\end{rem}
\begin{examp}\label{exmp:I.ub.10} If $U$ is an u.b. semiring and $I$ any set, then the Cartesian product  $U^I$ of copies of $U$
indexed by $I$ is a again u.b., and thus  every subsemiring $R$ of
$U^I$ also is u.b.
\end{examp}
We call such a semiring $R$ a \textbf{semiring of $U$-valued functions}. \pSkip

For similar reasons  the polynomial semiring $U[\lm_1, \dots, \lm_n]$ in any set if variables $(\lm_1, \dots, \lm_n)$ over an u.b. semiring $U$ is also u.b., cf.
\cite[Proposition 11.8]{IzhakianKnebuschRowen2009Valuation}.

\section{Companions on a free module}\label{sec:I.3}

We allow $R$ to be any semiring, but now always assume that $V$ is
a \textit{free} $R$-module with a fixed base $(\veps_i \ds |i\in
I),$ and that $q$ is a functional quadratic form on $V.$

\begin{defn}\label{defn:I.3.1} For any $(i,j)\in I\times I$, let
$C_{i,j}(q)$ denote the subset of $R$ consisting of the values
$b(\veps_i,\veps_j)$   for all companions $b$ of $q$ on
$R\veps_i\times R\veps_j$, cf.~Definition \ref{defn:I.2.1}. We call
the family of subsets of $R$
\begin{equation}\label{eq:I.3.1}
C(q):=(C_{i,j}(q) \ds |(i,j)\in I\times I)\end{equation} the
\bfem{companion table} of $q.$ Notice that
$C_{i,j}(q)=C_{j,i}(q).$ In the case $I=\{1,2,\dots,n\},$ we
usually write $C(q)$ as a symmetric $n\times n$-matrix,
\begin{equation}\label{eq:I.3.2}
C(q) :=\begin{pmatrix} C_{1,1}(q) &\dots & C_{1,n}(q)\\
\vdots & \ddots &\vdots\\
C_{n,1}(q) &\dots & C_{n,n}(q)\end{pmatrix},\end{equation} and
also speak of the \bfem{companion matrix} of $q.$  (Of course, if
$I$ is infinite, we may again view $C(q)$ as an $I\times I$-matrix
after choosing a total ordering on $I.$)
\end{defn}

As in \S\ref{sec:I.2}, we choose a fixed companion $b_0$ of $q.$ We put
\begin{equation}\label{eq:I.3.3}
\bt_{i,j}^0:=b_0(\veps_i,\veps_j).\end{equation}

Looking at Example \ref{examp:I.2.5}, we obtain the following
description of $C_{i,j}(q).$

\begin{prop}\label{prop:I.3.2} $C_{i,j}(q)$ is the set of all
$\bt\in R$ with
\begin{equation}\label{eq:I.3.4}
\lm^2q(\veps_i)+\mu^2q(\veps_j)+\lm\mu\bt=\lm^2q(\veps_i)
+\mu^2q(\veps_j)+\lm\mu\bt_{i,j}^0,\end{equation} for all
$\lm,\mu\in R\sm\{0\}.$\end{prop}

\begin{thm}\label{thm:I.3.3} If $(\bt_{i,j} \ds |(i,j)\in I\times I)$ is
a family in $R$ with $\bt_{i,j}\in C_{i,j}(q)$ and
$\bt_{i,j}=\bt_{j,i},$ then the bilinear form $b$ with
$b(\veps_i,\veps_j)=\bt_{i,j}$ is a companion of $q.$ This
establishes a bijection of the set of these families $(\bt_{i,j})$
with the set of companions of $q.$\end{thm}
\begin{proof} Given such a family $(\bt_{i,j}),$ let $b$ denote
the (unique) symmetric bilinear form with
$b(\veps_i,\veps_j)=\bt_{i,j}.$ Then $b$ is a  companion of
$q$ on the set $\bigcup_{i,j}R\veps_i\times R\veps_j$ by
definition of the sets $C_{i,j}(q).$ We conclude by
\propref{prop:I.2.8} that $b$ is a companion of $q$ (on $V$). The
last assertion of the theorem is now obvious.\end{proof}

\begin{cor}\label{cor:I.3.4}
$C_{i,j}(q)$ is the set of values $b(\veps_i,\veps_j)$ with $b$
running through all companions  of $q.$\end{cor}

\begin{proof} We choose a total ordering on $I$ and fix a pair
$(i,j)\in I\times I$ with $i\le j.$ Given an element $\bt_{i,j}$
of $C_{i,j}(q)$, we have to find a companion $b$ of $q$ with
$b(\veps_i,\veps_j)=\bt_{i,j}.$ By \thmref{thm:I.3.3} this is easy:
We choose for every $(k,\ell)\in I\times I$ with $k\le \ell$ and
$(k,\ell)$ different from $(i,j)$ an element $\bt_{k,\ell}$ of
$C_{k,\ell}(q).$ By the theorem, there exists a (unique) companion
$b$ of $q$ such that $b(\veps_k,\veps_\ell)=\bt_{k,\ell}$ for all
$(k,\ell)\in I\times I$ with $k\le \ell.$ In particular,
$b(\veps_i,\veps_j)=\bt_{i,j}.$\end{proof}

As a consequence of this corollary, we state

\begin{prop}\label{prop:I.3.5}
$q$ is rigid iff every set $C_{i,j}(q)$ consists of only one
element.
\end{prop}

\begin{remark}\label{rem:I.3.6}
It follows from \propref{prop:I.3.2} that for any $(i,j)$ in $I\times
I$ \begin{equation}\label{eq:I.3.5} C_{i,j}(q)=C_{i,j}(q \ds
|R\veps_i+R\veps_j).\end{equation} More generally,
$C_{i,j}(q)=C_{i,j}(q|W)$ for any free submodule $W$ of $V$ which
contain $\veps_i,\veps_j$ as part of a base.\end{remark} In the
case $I=\{1,\dots,n\},$ \propref{prop:I.3.2} and \thmref{thm:I.3.3} read
as follows:

\begin{schol}\label{schol:I.3.7}
Assume that $(\veps_i \ds |1\le i\le n)$ is a base of the free
module $V.$ Assume further that in this base $q$ has the
triangular
scheme
$$q=\begin{bmatrix} \al_{1,1} & \dots & \al_{1,n}\\
&\ddots & \vdots\\
0 & & \al_{n,n}\end{bmatrix}.$$ We can choose for $b_0$ the
balanced companion of $q$ with diagonal coefficients
% \note{Why did the notation switch from $a$ to $\al$?}
$2\al_{i,i}$ and upper
diagonal coefficients $\al_{i,j},$ cf.~\eqref{eq:I.1.23}. Thus the
companions of $q$ are the bilinear forms
$$b=\begin{pmatrix} \bt_{1,1} & \dots & \bt_{1,n}\\
\vdots & \ddots  &\vdots\\
\bt_{n,1} & \dots & \bt_{n,n}\end{pmatrix}$$ with coefficients
$\bt_{i,j}=\bt_{j,i}$ satisfying
\begin{equation}\label{eq:I.3.6}
(\lm+\mu)^2\al_{i,i}=(\lm^2+\mu^2)\al_{i,i}+\lm\mu\bt_{i,i},\end{equation}
for $1\le i\le n,$ and
\begin{equation}\label{eq:I.3.7}
 \lm^2
\al_{i,i}+\mu^2\al_{j,j}+\lm\mu\al_{i,j}=\lm^2\al_{i,i}+\mu^2\al_{j,j}+\lm\mu\bt_{i,j},\end{equation}
%\underline{}\note{I think $2\lm\mu\bt_{i,j}$ instead of $\lm\mu\bt_{i,j}$}
for
$1\le i<j\le n,$ with both $\lm,\mu$ running through
$R\sm\{0\}.$
\end{schol}

%\subsection{Upper bound semirings}

Our main focus will be a precise description of the sets
$C_{i,j}(q)$ in the case that $R$ is a tangible supersemifield.
%(cf. \S\ref{sec:I.4}).

But some facts  hold in the much more broader class of u.b.
semirings introduced in \S\ref{sec:I.ub}. \emph{In all the
following, up to the end of the paper the unsubscripted  notation
$\leq$ refers to the minimal order $\leq_R$ of the u.b. semiring $R$
under consideration. }

\begin{prop}\label{prop:I.3.8} If $R$ is u.b., then every set
$C_{i,j}(q)$ is convex in $R.$ In other words, if $\bt_1\le \gm\le
\bt_2$ and $\bt_1,\bt_2\in C_{i,j}(q),$ then $\gm\in
C_{i,j}(q).$\end{prop}

\begin{proof} We invoke \propref{prop:I.3.2}. For any $\lm,\mu
\in R\sm \{0\}$
$$
\begin{array}{lll}
\lm^2q(\veps_i)+\mu^2q(\veps_j)+\lm\mu\bt_1 & \le &
\lm^2q(\veps_i)+\mu^2q(\veps_j)+\lm\mu\gm
\\[1mm] & \le &  \lm^2q(\veps_i) +\mu^2q(\veps_j)+\lm\mu\bt_2.\end{array}
$$ Now the first sum equals the third
sum, and hence equals also the second sum.
\end{proof}

In the same vein we obtain
\begin{prop}\label{prop:I.3.9} Assume that $R$ is a semiring and
$q:V\to R$ is quasilinear on a set $T\subset V\times V$ (cf.
Definition \ref{defn:I.2.3}). Let $b_1$ and $b_2$ be symmetric
bilinear forms on $V.$
\begin{enumerate}
\item[a)] If $b_1$ and $b_2$ are companions of $q$ on $T,$ then
$b_1+b_2$ is a companion of $q$ on $T.$ \pSkip

\item[b)] Assume that $R$
is u.b. If $b_1+b_2$ is a companion of $q$ on $T,$ then both $b_1$
and $b_2$ are companions of $q$ on $T.$\end{enumerate}\end{prop}

\begin{proof}
Let $(x,y)\in T.$ We have $q(x+y)=q(x)+q(y).$

 a): From $q(x)+q(y)=q(x)+q(y)+b_i(x,y)$ $(i=1,2),$ we obtain
$$q(x)+q(y)=q(x)+q(y)+b_2(x,y)=q(x)+q(y)+b_1(x,y)+b_2(x,y).$$

 b): We have for $i=1,2$
 $$q(x)+q(y)\le q(x)+q(y)+b_i(x,y)\le
 q(x)+q(y)+b_1(x,y)+b_2(x,y).$$ The first sum equals the third sum,
 and
 hence also equals the second sum.
\end{proof}

Let $R_0$ denote the prime supertropical semiring, $R_0=\{0,1,e\}$.

\begin{cor}\label{cor:I.3.10} Assume that $R$ is u.b. and contains
$R_0$ as a subsemiring. Let~ $V$ be any $R$-module. Then
a symmetric bilinear form $b$ on $V$ is a companion of a
quasilinear form  $q: V\to R$ iff $eb$ is a companion of
$q.$\end{cor}

\begin{prop}\label{prop:I.3.11} Assume again that $R_0 \subset R$ and further that %$f$ is an element of $R$ with $f+f=f$ (e.g., $R$ contains
% $R_0=\{0,1,e\}$ and $f=e$).
%Assume furthermore that
$C_{i,j}(q)\cap Re \ne\emptyset.$ Then $C_{i,j}(q)$ is closed under
addition.\end{prop}

\begin{proof}
We again invoke \propref{prop:I.3.2}. Let $\bt_0 \in C_{i,j}(q) \cap
Re.$ Pick $\bt_1,\bt_2\in C_{i,j}(q).$ For any $\lm,\mu\in R$ the
three sums $\lm^2q(\veps_i)+\mu^2q(\veps_j)+\lm \mu \bt_k$
$(k=0,1,2)$ are equal. From this we obtain that
$$
\begin{array}{lll}
\lm^2q(\veps_i)+\mu^2q(\veps_j)+\lm\mu(\bt_1+\bt_2) & = &
\lm^2q(\veps_i) +\mu^2q(\veps_j)+\lm\mu(\bt_0+\bt_0)
\\[1mm] &  = & \lm^2q(\veps_i)+\mu^2q(\veps_j)+ \lm\mu\bt_0,
\end{array}
$$ since $\bt_0+\bt_0=\bt_0.$ We conclude that
$\bt_1+\bt_2\in C_{i,j}(q).$
\end{proof}

\begin{prop}\label{prop:I.3.12} %Assume that   $R$ is u.b., and that $(a+b)^2=a^2+b^2$ for any $a,b\in R.$
%In particular  $(a=b=1)$, $2\cdot 1_R=4\cdot 1_R$. (For
%example, $R$ could be u.b., containing the prime supertropical
%semiring $R_0).$ Then, with $e:=2\cdot 1_R,$
If $R$ is a supertropical semiring, then
\begin{equation}\label{eq:I.3.9}
C_{i,i}(q)=\{\bt\in R \ds |0\le\bt\le eq(\veps_i)\}\end{equation}
for every $i\in I.$\end{prop}

\begin{proof} Due to Remark \ref{rem:I.3.6}, we may assume that
$I=\{1\},$ $V=R\veps_1,$ and then may simplify to $V=R, $
$\veps_1=1_R.$ Let $\al:=q(\veps_1).$ This means that $q(x)=\al
x^2$ $(x\in R).$ We have
$$q(x+y)=\al x^2+\al y^2=\al x^2+\al y^2+e\al xy.$$
Thus both linear forms $b_0(x,y):=0,$ $b_1(x,y):=e\al xy$ are
companions of $q,$ i.e., $0\in C_{1,1}(q) $ and $e\al\in
C_{1,1}(q).$
 We know by \propref{prop:I.3.8} that $C_{1,1}(q)$ is
convex; hence $C_{1,1}(q)$ contains the interval
$$[0,e\al]:=\{\bt\in R \ds |0\le \bt\le e\al\}.$$
On the other hand, if $b$ is any companion of $q,$ then
$$eq(\veps_1)=q(e\veps_1)=q(\veps_1+\veps_1)=eq(\veps_1)+b(
\veps_1,\veps_1);$$ hence $b(\veps_1,\veps_1)\le eq(\veps_1).$
Trivially, $0\le b(\veps_1,\veps_1).$ Thus $b(\veps_1,\veps_1)$ is
contained in $[0,e\al].$
\end{proof}

\begin{thm}\label{thm:I.3.13}
Assume that $R$ is a supertropical semiring with $e \tT = \tG.$
Assume further that $b: V\times V\to R$ is a symmetric bilinear
form that accompanies $q$  on the set $\bigcup_{i\in I} \tT
\veps_i.$ Then $b$ is a companion of $q.$\end{thm}

\begin{proof} By \propref{prop:I.2.8} it suffices to verify that $b$
accompanies $q$ on $R\veps_i\times R\veps_j$ for any two indices
$i,j\in I.$ This means that we have to verify (cf.~\eqref{eq:I.2.4})
\begin{equation}\label{eq:I.3.10}
\lm^2q(\veps_i)+\mu^2q(\veps_j)+\lm\mu
b(\veps_i,\veps_j)=\lm^2q(\veps_i)+\mu^2q(\veps_j)+\lm\mu
b_0(\veps_i,\veps_j),\end{equation} for any $\lm ,\mu\in
R\sm\{0\},$ knowing that \eqref{eq:I.3.10} holds for
$\lm,\mu\in \tT.$ We start with \eqref{eq:I.3.10}  for fixed
$\lm,\mu\in \tT.$ Adding $\lm^2q(\veps_i)+\lm\mu
b(\veps_i,\veps_j)$ on both sides, we obtain \eqref{eq:I.3.10} with~
$\lm$ replaced by~ $e\lm.$ In an analogous way we obtain
\eqref{eq:I.3.10} with $\mu$ replaced by $e\mu.$ Multiplying
\eqref{eq:I.3.10} by the scalar $e$, we obtain \eqref{eq:I.3.10} with
$\lm,\mu$ replaced by $e\lm,e\mu.$ The claim is proved.\end{proof}

\begin{remark}\label{rem:I.3.14}
As the proof reveals, the theorem holds more generally if $V$ is
any $R$-module and $(\veps_i \ds |i\in I)$  generates
$V.$\end{remark}

\section{The companions of a quadratic form over a tangible supersemifield}\label{sec:I.4}

In this section, we study quadratic forms and their companions over
a  tangible supersemifield, in order to understand quasilinear-rigid
decompositions.

In many arguments we can retreat from a supersemifield~$R$ to a
tangible supersemifield~$R'$ (cf. Remark \ref{rem:I.4.3}) by omitting the ``superfluous'' ghost
elements of ~$R.$ But for categorical reasons we do not always
exclude non-tangible supersemifields from our study.

As in the classical quadratic form theory of fields the
``\textit{square classes}'' of a semiring $R$ will be important
for understanding quadratic forms over $R.$

\begin{defn}\label{defn:I.4.4} We call two elements $x,y$ of a semiring $R$ \bfem{square equivalent},
and write $x\sim_{\sq} y,$ iff there exists a unit $z$ of $R$ \footnote{Then certainly $z \in \tT(R).$} with
$xz^2=y.$  The set $x(R^*)^{2}$
consisting of all $y\in R$ with $x\sim_{\sq}y$ will be called the
\bfem{square class} of $x$ (in $R$). \end{defn}

\begin{remarks}\label{rem:I.4.5}
Assume that $R$ is a supertropical semiring.

\begin{enumerate}
  \item [a)]   Every square class
different from~$\{0\}$ consists either of elements of $\tT(R)$ or
of elements of $\tG(R).$ \pSkip

  \item [b)]
  If two elements of $eR$ are square equivalent in $R$, then they
are square equivalent in~$eR.$ \pSkip

  \item [c)]
 If $R$ is a tangible supersemifield, then the square class of
any $x\in R$ is the set $x(\tT(R))^2.$ Moreover two elements of $eR$, which are square equivalent in $eR$, are square equivalent  in $R.$
\end{enumerate}
 \end{remarks}

\begin{defn}\label{defn:I.4.6} A
\bfem{multiquadratic extension} of a supersemifield~$R$ is a
supersemifield~$R'$ containing~$R$ as a subsemiring such that
$x^2\in R$ for every $x\in R'.$ We call  such  an extension $R' \supset R$ \textbf{\textbf{full}} if $R = \{ x^2 \ds | x \in R ' \}.$ \end{defn}

Given any supertropical semifield $R$, we now indicate how to construct a full multiquadratic  extension   $R' \supset R$. We may  assume that $R$ is a standard supertropical semifield
$$R = \STR(\tT, \tG , \nu) = \tT \sqcup \tG \sqcup \{ 0 \}$$
arising from a group epimorphism $\nu: \tT \twoheadrightarrow \tG$, cf. Example \ref{stropi}. Recall that $\tG = \tG(R)$ is totally ordered.

\begin{lem}\label{lem:I.4.7} If $x,y \in \tG$ and $x < y$ then $x^n < y^n$ for every $n \in \N$.
\end{lem}

\begin{proof} An easy induction on $n$. Assuming that $n \leq 2$ and $x^{n-1} < y^{n-1}$, we obtain
$$ x^n = x^{n-1} x < x^{n-1}y < y^{n-1} y =  y^n.$$
\end{proof}

We take the divisible hull $\tG_\dv = \Q \otimes_Z \tG.$ Using multiplicative notation we write $\frac 1 n \otimes x = \sqrt[n]{x}$ for $x \in \tG$, $n \in \N.$   Then
$\frac m n \otimes x = \sqrt[n]{x^m}$ for all $m \in \Z$. The total ordering of $\tG$  extends to a total ordering of the group $\tG_\dv$ by the rule
$$ \sqrt[n]{x^m} \leq \sqrt[n]{y^r}  \dss{\Leftrightarrow} x^m \leq y ^ r,$$
as is easily  verified. We focus on the ordered subgroup
$$ \tG^{\hlf}:= \hlf \otimes \tG = \{ z \in \tG_\dv \ds | z^2 \in \tG \} $$
and note that the square map
$$ \tG^{\hlf} \to \tG, \qquad z \mapsto z^2, $$
is an isomorphism of ordered ablian groups with inverse map $z \mapsto \sqrt{z}.$

We further construct an ablian group $\tT' \supset \tT$ with $\tT' / \tT$ of exponent $2$ as follows. We choose a presentation  $\tT = V / U $ with $V$ a free ablian group and $U$ a subgroup of $V$, which then is well known to be also free, cf. e.g. \cite[p. 143]{Kur} or \cite[Thm. 14.5]{Fuchs}. We then  put $\tT': = V' / U$ with
$$ V' :=  \hlf  \otimes V = \{ x \in V_\dv \ds | x^2 \in  V \}. $$
Clearly $\tT$ is a subgroup of $\tT'$ and the square map $\tT' \overset{2}{\longrightarrow} \tT$ is a group epimorphism.

We obtain a commuting square of group epimorphisms
\begin{equation}\label{eq:str}
    \begin{gathered}
    \xymatrix{
  \tT' \ar@{->>}[r]^{\nu'} \ar@{->>}[d]_{2} &  \tG^{\hlf} \ar@{->}[d]^{2}_{\cong}\\
  \tT \ar@{->>}[r]^{\nu} &  \tG}
    \end{gathered} \tag{$*$}
    \end{equation}
by defining $\nu'(z) : = \sqrt{\nu(z^2)}$ for $z \in \tT'$.

The homomorphism $\nu': \tT' \to \tG^{\hlf}$ restricts to   $\nu: \tT \to \tG$, and  thus the supersemifield
$$R' = \STR(\tT', \; \tG^{\hlf} , \; \nu') = \tT' \sqcup \tG^{\hlf} \sqcup \{ 0 \}$$
contains $R$ as a subsemiring. It is now clear from  $(*) $ and the rules in Example \ref{stropi} that $R'$ is a full multiquadratic extension of $R$.

It can be shown that in this way we obtain a cofinal system of \emph{all} full multiquadratic extension of $R$ by choosing different presentations $\tT = V/U$, but we do not need this result at present.

\begin{rem}\label{rem:I.4.5}
The kernel of  $\tT(R') \overset{2}{\longrightarrow} \tT(R)$ for a full multiquadratic extension $R' \supset R$ can be very big. Indeed, starting with  $R' = \STR(\tT', \tG^{\hlf} , \nu')$ as just constructed,  we can form
$$R'' = \STR(H \times \tT', \;  \tG^{\hlf} , \; \nu'') $$
with $H$ any abelian group of exponent $2$ and $\nu'' = \nu' \circ p$, where $p$ is the obvious projection from  $H \times \tT'$ to $\tT'$. Then $R''$ is again a full multiquadratic extension of $R$ and
$$ \ker(\tT(R'') \overset{2}{\longrightarrow} \tT(R)) = H \times \ker(\tT(R') \overset{2}{\longrightarrow} \tT(R)). $$

\end{rem}
\pSkip
\emph{Until the end of this section, \textbf{we fix a tangible
supersemifield} $R.$ We discard the ``trivial'' supersemifields where
$\tG(R)=\{e\}$.}

\begin{notations}\label{notat:I.4.7}
We write $\tT:=\tT(R),$ $\tG:= \tG(R).$ We choose a full multiquadratic
extension of~$R,$ denoted by $R^{1/2},$ and write $\tT^{1/2}:=\tT(R^{1/2}),$
$\tG^{1/2}:=\tG(R^{1/2}).$ Notice  that for every
$x\in\tG$ there is a unique $z\in \tG^{1/2}$ with $z^2=x.$
We denote this unique square root $z$
of~$x\in\tG$ by $\sqrt x.$ (Elements of $\tT$ may have different
square roots in~$\tT^{1/2}.)$\end{notations}

At present the extension $R^{1/2} $ of $R$ will only serve
as a tool to ease notation. Later this will change; then we will
need a precise theory of multiquadratic extensions, to be given in
a sequel of this paper.

\begin{terminology}\label{term:I.4.8}
\quad{}
\begin{enumerate}
    \item[a)] Concerning the totally ordered group $\tG$, we have a
dichotomy:

Either $\tG$ is \bfem{densely ordered}, i.e., for any two elements
$x_1<x_2$ of $\tG$ there exists some $y\in \tG$ with $x_1<y<x_2.$
Then also $eR=\tG\cup\{0\}$ is densely ordered, since $\tG$ does
not have a smallest element.

Or $\tG$ is \bfem{discrete}, i.e., for every $x\in\tG$ there
exists a biggest $x'\in\tG$ with $x'<x.$
For short, we say that  $R$ \bfem{is dense} and
that $R$ \bfem{is discrete}, respectively. \pSkip

\item[b)] If $R$ is discrete, we fix some $\pi\in \tT$ such that
$e\pi$ is the biggest element $z$ of $eR$ with $z<e,$ and call
$\pi$ a \bfem{prime element} of $R.$ \pSkip

\item[c)] Notice that if $R$ is dense, then $R^{1/2}$ also is
dense; while if $R$ is discrete, then $R^{1/2} $ is discrete. In
the discrete case, $\sqrt{e\pi}$ is the biggest element of
$\tG^{1/2}$ smaller than $e.$ Then we choose an element $z$
of $R^{1/2}$ with $z^2=\pi,$ and denote this element $z$ by
$\sqrt{\pi}.$ It is a prime element of $R^{1/2}.$
\end{enumerate}

\end{terminology}

If $V$ is a free module over the tangible supersemifield $R$ with
base $(\veps_i \ds |i\in I)$, and $q:V\to R$ is a quadratic form
on $V,$ we want to determine the sets $C_{i,j}(q)$ defined in
\S\ref{sec:I.3}. Recall that $ C_{i,j}(q)$ coincides with $C_{i,j}(q
\ds| R\veps_i+R\veps_i),$ as observed in \S\ref{sec:I.3} (Remark
\ref{rem:I.3.6}); hence it suffices to study the case $I=\{1,2\}$.

Thus we now focus on the following case: $V$ is free with base
$\veps_1,\veps_2,$ and
\begin{equation}\label{eq:I.4.1}
q=\begin{bmatrix} \al_1 &\al\\
&\al_2\end{bmatrix}\end{equation} for given $\al_1,\al_2,$
$\al\in R.$ (More precisely, $q$ is the functional quadratic
form on $V$ represented by this triangular scheme.)

We know already (\propref{prop:I.3.12}) that
\begin{equation}\label{eq:I.4.2}
C_{i,i}(q)=[0,e\al_i]\qquad (i=1,2).\end{equation} Here, as in
\S\ref{sec:I.3}, we use interval notation with respect to the
minimal ordering of $R;$ namely for $c\in R$
$$[0,c]:=\{x\in R \ds| 0\le x\le c \}.$$

Our problem is to determine $C_{1,2}(q).$ Certainly $\al\in
C_{1,2}(q),$ cf. \propref{prop:I.1.12}.

\begin{lemma}\label{lem:I.4.9}
$C_{1,2}(q)$ is the set of all $\bt\in R$ with
\begin{equation}\label{eq:I.4.3}
\lm\al_1+\lm^{-1}\al_2+\al \ds
=\lm\al_1+\lm^{-1}\al_2+\bt\end{equation} for every $\lm\in
\tT.$\end{lemma}

\begin{proof} By Scholium \ref{schol:I.3.7} and \thmref{thm:I.3.13},
we know that $C_{1,2}(q)$ is the set of all $\bt\in R$ with
$$\lm^2\al_1+\mu^2\al_2+\lm\mu\al \ds = \lm^2 \al_1 +\mu^2\al_2+\lm\mu\bt$$
for all $\lm,\mu\in \tT.$ Dividing by $\lm\mu$, we obtain
condition \eqref{eq:I.4.3}, there with $\lm/\mu$ instead of $\lm;$
hence the claim.
\end{proof}

\begin{prop}\label{prop:I.4.10}
Assume that $\al_1=0$ or $\al_2=0.$ Then
$C_{1,2}(q)=\{\al\}.$\end{prop}

\begin{proof}
We may assume that $\al_2=0.$ Now condition \eqref{eq:I.4.3} reads
$$\lm\al_1+\al=\lm\al_1+\bt.$$
If this holds for every $\lm\in \tT,$ then $\al=\bt,$\footnote{Here we need the nontriviality assumption that $\tG \neq\{e \}$. } and we
conclude by \lemref{lem:I.4.9} that $C_{1,2}(q)=\{\al\}.$\end{proof}

In the following, we use the ``\textbf{$\nu$-notation}''.
For $a,b\in R$,
we say that $a$ is $\nu$-\textbf{equivalent} to $b$, and write
$a\cong_\nu b$  if $ea = eb$. We say that $a$ is  $\nu$-\textbf{dominated} by  $b$ (resp. \textbf{strictly} $\nu$-\textbf{dominated}) by $b$, and write $a \nule b$ (resp. $a \nul b$), if $ea \leq eb$ (resp. $ea < eb$).
We say that $a\in R$ is a $\nu$-\textbf{square} (in $R$),
iff $a\cong_\nu b^2$ for some $b\in R.$ \pSkip

We now assume that $\al_1\ne0$ and $\al_2\ne0.$

\begin{convention}\label{conv:I.4.11}
{}\

\begin{enumerate}\item[a)] If $\al_1\al_2$ is a $\nu$-square, we choose some
$\elm \in \tT$ with $\al_1\elm^2\cong_\nu \al_2.$ Otherwise we
choose $\elm \in \tT^{1/2}$ with $\al_1\elm^2\cong_\nu \al_2.$  Thus $\al_1\elm \cong_\nu \elm^{-1}\al_2$ (in
$R^{1/2})$ in both
cases. If $\al_1, \al_2 \in \tT$, we may think of $\elm \al_1 $ as a sort of ``tangible geometric
mean'' of $\al_1$, $\al_2$, since $e\elm \al_1 =\sqrt{e\al_1\cdot
e\al_2}.$ \pSkip

 \item[b)] We distinguish
the following three cases.
\begin{description}
\item[Case I] $\elm \in \tT,$ i.e., $\al_1\al_2$  is a $\nu$-square. \pSkip

\item[Case II] $R$ is dense, $\elm \notin \tT.$ \pSkip

\item[Case III]$R$ is discrete,  $\elm \notin \tT.$ \pSkip
\end{description}

 \item[c)]
In Case III, we choose $\sig,\tau\in \tT$ with $e\tau<e\elm <e\sig$
and with no element of $\tG$ between~$e\tau$ and
$e\sig.$ In other terms, employing the prime element $\pi$ of $R,$
$$\tau\cong_\nu \pi\sig,\qquad \elm \cong_\nu\sqrt{\pi}\sig.$$
Replacing $\sig$ by $\elm ^2\tau^{-1}$, we may assume in addition
(for simplicity) that $\sig
\tau=\elm^2.$\end{enumerate}\end{convention}

\begin{thm}\label{thm:I.4.12}
Assume that $\al^2>_\nu \al_1\al_2.$ Then $C_{1,2}(q)=\{\al\},$
except in the following case: $R$ is discrete, $\al_1\al_2$ is not
a $\nu$-square, and $\al^2\cong_\nu \pi^{-1}\al_1\al_2.$ Then, if
$\al_1\in \tG$ and $\al_2\in\tG,$
\begin{equation}\label{eq:I.4.4}
C_{1,2}(q)=[0,e\al],\end{equation} while, if $\al_1\in \tT$ or
$\al_2\in \tT,$
\begin{equation}\label{eq:I.4.5}
C_{1,2}(q)=\{\bt\in R \ds| \bt\cong_\nu\al\}.\end{equation}
\end{thm}

\begin{proof} The case $\al_1\al_2=0$ was covered by
\propref{prop:I.4.10}. Assume now that $\al_1\al_2\ne0.$ We again
rely on \lemref{lem:I.4.9}. If we are in Case I, we insert $\lm=\elm $
in condition \eqref{eq:I.4.3}. Since $\elm \al_1\cong_\nu\elm^{-1}
\al_2$ and $\al>_\nu\elm \al_1,$ we obtain
$$\al=e\elm \al_1+\al\overset!= e\elm \al_1+\bt.$$
This forces $\bt=\al;$ hence $C_{1,2}(q)=\{\al\}.$ If we are in
Case II or III, then no $\lm\in \tT$ is $\nu$-equivalent to
$\elm .$ If $\lm>_\nu\elm $ then $\lm\al_1>_\nu\lm^{-1}\al_2,$ and
condition \eqref{eq:I.4.3} reads
\begin{equation}\label{eq:I.4.6} %\renewcommand{\theequation}{*}\addtocounter{equation}{-1}\label{eq:str.1}
\lm\al_1+\al=\lm\al_1+\bt.\end{equation} If $\lm<_\nu\elm $ then
$\lm\al_1<_\nu\lm^{-1}\al_2,$ and \eqref{eq:I.4.3} reads
\begin{equation}\label{eq:I.4.7}
\lm^{-1}\al_2+\al=\lm^{-1}\al_2+\bt.\end{equation} Assume now that
we are in Case II. Since $R$ is dense, we can choose $\lm\in \tT$
with
$$\elm \al_1<_\nu\lm\al_1<_\nu\al.$$ For such $\lm$
condition \eqref{eq:I.4.6} reads $\al=\lm\al_1+\bt.$ This forces
$\bt=\al,$ and we conclude again that $C_{1,2}(q)=\{\al\}.$

We are left with Case III. Now $\al^2 \ge_\nu \pi^{-1} \al_1 \al_2$, since $e \al_1 \al_2$ is not a square.
We have $\al_1 \al_2 \cong_\nu \elm^2 \al_1 \cong_\nu \pi \sig^2 \al_1^2$; hence $\al^2 \ge_\nu \sig^2 \al_1$, i.e. $\al\ge_\nu\sig\al_1.$
If
$\al>_\nu\sig\al_1,$ i.e., $\al^2>_\nu\pi^{-1}{\al_1\al_2},$ we
insert $\lm=\sig$ into condition \eqref{eq:I.4.6} and obtain
$$\al=\sig\al_1+\al\overset!=\sig\al_1+\bt.$$
This forces $\bt=\al;$ hence $C_{1,2}(q)=\{\al\}.$

%\note{We already made this assumption, so it is superfluous}
Assume finally that  $\al\cong_\nu  \sig\al_1,$
i.e., $\al^2\cong_\nu \pi^{-1}\al_1\al_2.$ Inserting $\lm=\sig$
into \eqref{eq:I.4.6}, we obtain
$$e\sig\al_1=\sig\al_1+\al\overset!=\sig\al_1+\bt.$$
If $\al_1\in \tT,$ this forces
$\bt\cong_\nu\sig\al_1\cong_\nu\al;$ while if $\al_1\in\tG, $ this
only forces $\bt\le_\nu\sig\al_1\cong_\nu\al.$

Thus we have found the   constraints for $\bt\in C_{1,2}(q)$ that
$\bt\cong_\nu\al$ if $\al_1\in \tT,$ resp. $\bt\le_\nu\al$ if
$\al_1\in\tG.$ Of course, these constraints are also valid if
$\al_2\in \tT$, resp. $\al_2\in\tG,$ since we may interchange the
base vectors $\veps_1$ and $\veps_2$ of $V.$

If $\bt\cong_\nu \al,$ it is easily checked that \eqref{eq:I.4.6}
holds for all $\lm\in \tT$ with $\lm\ge\sig,$ and \eqref{eq:I.4.7}
holds for all $\lm\in \tT$ with $\lm\le \tau.$ Thus
$$C_{1,2}(q)=\{\bt\in R \ds| \bt\cong_\nu\al\}$$
if $\al_1\in \tT$ or $\al_2\in \tT.$ If both $\al_1$ and $\al_2$
are ghost and $\bt\le_\nu\al,$ again an easy check reveals that
\eqref{eq:I.4.6} holds for all $\lm \in \tT$ with $\lm\ge \sig,$ and
\eqref{eq:I.4.7} for all $\lm\in \tT$ with $\lm \le\tau.$ We
conclude that now $C_{1,2}(q)=[0,e\al].$
\end{proof}

\begin{thm}\label{thm:I.4.13}

  Assume that $\al^2\le_\nu\al_1\al_2.$

\begin{enumerate}\item[a)] Then
\begin{equation}\label{eq:I.4.8}
C_{1,2}(q)=\{\bt\in R \ds |\bt^2\le_\nu\al_1\al_2\}
\end{equation} except
%\note{Doesn't 4.8 imply 4.9 here?}
in the case that $R$ is discrete,
$\al_1\al_2$ is not a $\nu$-square and both $\al_1$ and $\al_2$
are ghost. Then
\begin{equation}\label{eq:I.4.9}
C_{1,2}(q)=\{\bt\in R
\ds|\bt^2\le_\nu\pi^{-1}\al_1\al_2\}.\end{equation}\pSkip

\item[b)] Using Convention \ref{conv:I.4.11}, this means more
explicitly: \pSkip

In Case I: $C_{1,2}(q)=[0,e\elm \al_1]=[0,e\elm^{-1}\al_2].$ \pSkip

In Case II: $C_{1,2}(q)=[0,e\elm \al_1[:=\{\bt\in
R\ds|0\le_\nu\bt<_\nu e\elm \al_1\}.$ \pSkip

In Case III: $C_{1,2}(q)=[0,e\tau\al_1]=[0,e\sig^{-1}\al_2]$ if $\al_1\in \tT$ or $\al_2\in \tT,$
whereas $C_{1,2}(q)=[0,e\sig\al_1]=[0,e\tau^{-1}\al_2]$ if $\al_1\in\tG$ and $\al_2\in \tG.$
\end{enumerate}\end{thm}

\begin{proof}
Now $\al\le_\nu\elm \al_1 \cong_\nu\elm^{-1}\al_2.$ We again
exploit \lemref{lem:I.4.9}. If $\lm\in \tT$ and $\lm>_\nu\elm ,$ then
condition \eqref{eq:I.4.3} there reads
\begin{equation}\lm\al_1=\lm\al_1+\bt,\tag{\ref{eq:I.4.6}$'$}\end{equation}
and if $\lm\in T,$ $\lm<_\nu \elm $, then \eqref{eq:I.4.3} reads
\begin{equation}\lm^{-1}\al_2=\lm^{-1}\al_2+\bt.\tag{\ref{eq:I.4.7}$'$}\end{equation}

Assume first that we are in Case I. Inserting $\lm\in T$ with
$\lm\cong_\nu\elm $ into \eqref{eq:I.4.3}, we obtain for $\bt\in
C_{1,2}(q)$ the necessary condition
$$e\elm \al_1=e\elm \al_2+\bt,$$
which means that $\bt\le_\nu\elm \al_1.$ This constraint for $\bt$
implies that for $\lm{}\!\!>_\nu\elm $ we have $\lm\al_1>_\nu\bt;$
hence (\ref{eq:I.4.6}$'$) holds, while for $\lm<_\nu\elm $ we have
$\lm^{-1}\al_2>_\nu\bt,$ and hence (\ref{eq:I.4.7}$'$) holds. Thus
$\bt\le_\nu\elm \al_1$ implies that $\bt\in C_{1,2}(q).$ We
conclude that in Case I
$$C_{1,2}(q)=[0,e\elm \al_1]=[0,e\elm^{-1}\al_2].$$
In Case II $(R$ dense), condition (7.6$'$) for all $\lm\in \tT$
with $\lm>_\nu\elm $ forces $\bt\le_\nu\elm \al_1$, and hence $\bt<_\nu
\elm \al_1$, while in Case III we obtain the constraint $\bt\le_\nu
\sig\al_1.$

Assume that $\bt<_\nu\elm \al_1.$ Then (\ref{eq:I.4.6}$'$) holds for
all $\lm>_\nu\elm .$ If $\lm<_\nu\elm $ and $\lm\in \tT$ we have
$\lm^{-1}\al_2>_\nu\elm^{-1}\al_2=\elm \al_1>_\nu\bt;$ hence
(\ref{eq:I.4.7}$'$) is valid. Thus $\bt\in C_{1,2}(q).$ We conclude
that in Case II,
$$C_{1,2}(q)=[0,\elm \al_1[.$$

We turn to Case III and assume the constraint
$\bt\le_\nu\sig\al_1.$ For $ \lm>_\nu\sig$ condition (\ref{eq:I.4.6}$'$) is
evident. For $\lm \cong_\nu \sig$ condition (\ref{eq:I.4.6}$'$) gives the
constraint
$$\sig\al_1\overset!=\sig\al_1+\bt.$$
If $\al_1\in\tG,$ this holds. If $\al_1\in \tT,$ we obtain the
stronger constraint
$$\bt<_\nu\sig\al_1,\quad \text{i.e.,}\quad
\bt\le_\nu\tau \al_1.$$ Now assume that $\lm\in \tT$ and
$\lm<_\nu\elm .$ We have $\lm\le_\nu\tau,$ hence
$$\lm^{-1}\al_2 \ge_\nu\tau^{-1}\al_2\cong_\nu \sig\al_1\ge_\nu\bt.$$
If $\al_2\in\tG,$ this implies (\ref{eq:I.4.7}$'$); while if $\al_2\in \tT,$
we obtain the constraint
$\bt<_\nu\tau^{-1}\al_2\cong_\nu\sig\al_1,$ i.e., again
$\bt\le_\nu\tau\al_1.$ We conclude that if $\al_1\in\tG$ and
$\al_2\in\tG,$ then
$$C_{1,2}(q)=[0,\sig\al_1]=[0,\tau^{-1}\al_2];$$
while if $\al_1\in \tT$ or $\al_2\in \tT,$ then
$$C_{1,2}(q)=[0,e\tau\al_1]=[0,e\sig^{-1}\al_2].$$
We have proved part b) of the theorem. It is immediate that this
is equivalent to part a).
\end{proof}

We state three consequences of Proposition \ref{prop:I.4.10} and Theorems \ref{thm:I.4.12} and
\ref{thm:I.4.13}.

\begin{cor}\label{cor:I.4.14}
Let $\al_1,\al_2,$ $\al\in R.$ The quadratic form
$$q=\begin{bmatrix} \al_1 &\al\\
 &\al_2\end{bmatrix}$$ is quasilinear iff either
$\al^2\le_\nu\al_1\al_2,$ or $R$ is discrete where
$\al^2\cong_\nu\pi^{-1}\al_1\al_2$ and  both $\al_1,\al_2$ are
ghost.
\end{cor}

\begin{proof}
$q$ is quasilinear iff $q$ is quasilinear at $(\veps_1,\veps_2)$,
iff $0\in C_{1,2}(q).$ The assertion can be read off from Proposition \ref{prop:I.4.10} and Theorems
\ref{thm:I.4.12} and \ref{thm:I.4.13}. (Observe that if $R$ is discrete
and $\al^2\cong_\nu\pi^{-1}\al_1\al_2,$ then $\al_1\al_2$ is not a
$\nu$-square, and hence we are in Case III.)
\end{proof}

\begin{cor}\label{cor:I.4.15}
Assume that $\al_1,\al_2,$ $\al\in R$, where
$\al^2\le_\nu\al_1\al_2 $ if $R$ is dense, and
\\ $\al^2\le_\nu\pi^{-1}\al_1\al_2$ if $R$ is discrete. Then the
triangular schemes
$$\begin{bmatrix}
\al_1&\al\\
 &\al_2\end{bmatrix},\qquad\begin{bmatrix}\al_1 & e\al\\
& \al_2\end{bmatrix}$$ represent the same quadratic form on
$V=R\veps_1+R\veps_2.$
\end{cor}

\begin{proof}
Let $q$ be the functional form represented
by
$\left[\begin{smallmatrix}\al_1&\al\\
&\al_2\end{smallmatrix}\right] .$ The triangular scheme $\left[\begin{smallmatrix}\al_1&e\al\\
&\al_2\end{smallmatrix}\right] $ represents $q$ iff $e\al\in
C_{1,2}(q).$ The claim can again be read off from Theorems
\ref{thm:I.4.12} and \ref{thm:I.4.13}. (Observe that if $R$ is discrete,
then either $\al^2\le_\nu\al_1\al_2,$ or we are in Case III.)
\end{proof}

N.B. A big part of \corref{cor:I.4.15} is obvious from Corollaries
\ref{cor:I.4.14} and \ref{cor:I.3.10}.

\begin{cor}\label{cor:I.4.16}
Let $\al_1,\al_2,$ $\al\in R$ and
$q=\left[\begin{smallmatrix}\al_1&\al\\
&\al_2\end{smallmatrix}\right] .$
\begin{enumerate}\item[a)] When $R$ is dense, then $q$ is rigid at
$(\veps_1,\veps_2)$ iff $\al_1\al_2<_\nu\al^2.$ \pSkip
\item[b)] When $R$ is discrete, then $q$ is rigid at
$(\veps_1,\veps_2)$ iff $\al_1\al_2<_\nu\pi\al^2.$\end{enumerate}
\end{cor}

\begin{proof}
Browsing through Proposition \ref{prop:I.4.10} and Theorems \ref{thm:I.4.12} and \ref{thm:I.4.13}, we see
that \\ $C_{1,2}(q)=\{\al\} $ precisely in these cases.\end{proof}

\section{A closer study of $\Rig(q)$ }\label{sec:I.5}

Throughout this section, \textit{we assume that the semiring $R$ is
u.b.} (cf. Definition \ref{defn:I.ub.1}) and satisfies
Remark~\ref{assump}, and we examine the possible quasilinear-rigid
decompositions, cf.~Definition~\ref{defn:I.5.3}. We have a natural
partial ordering on the module of $R$-valued functions $\Quad(V)$ as
follows: If $q,q'\in\Quad(V)$, then
\begin{equation}\label{eq:I.5.14}
q\le q' \dss \Leftrightarrow \forall x\in V: q(x)\le q'(x).
\end{equation}
This ordering is compatible with addition and scalar multiplication.

The ordering \eqref{eq:I.5.14} leads to a new characterization of
quasilinear parts.

\begin{lemma}\label{lem:I.5.10}
If $q,q'$ are quasilinear forms on $V$ with $q\le q',$ then
$(q)_{\QL}\le (q')_{\QL}.$
\end{lemma}

\begin{proof}
For $x=\sum\limits_i x_i\veps_i\in V$ we have
$$q_{\QL}(x) \ds =\sum_ix_i^2q(\veps_i) \dss \le\sum_ix_i^2q'(\veps_i) \ds =q_{\QL}'(x),$$
due to formula \eqref{eq:I.5.3}.% describing $\ql=q_{\QL}.$
\end{proof}

\begin{thm}\label{thm:I.5.11}
If $q$ is any quadratic form on $V,$ then $q_{\QL}$ is the unique
maximal quasilinear form $\ql$ on $V$ with $\ql\le q;$ in more
explicit terms, $q_{\QL}\le q$ and $\chi\le q_{\QL}$ for every
quasilinear form $\chi\le q.$\end{thm}

\begin{proof}
For quasilinear  $\chi\le q$, we conclude from \lemref{lem:I.5.10}
that $\chi=\chi_{\QL}\le q_{\QL}.$
\end{proof}

We now look at rigid complements in terms of the ordering
\eqref{eq:I.5.14}.

\begin{thm}\label{thm:I.5.12}
Assume that $R$ is u.b. and satisfies Remark~\ref{assump}.
\begin{enumerate}
\item[a)] The $R$-module $\Rig(V)$ is a lower set in $\Quad(V),$ i.e., if $\chi$ and $\rg$ are quadratic forms on $V$ with $\rg$ rigid and $\chi\le\rg,$ then $\chi$ is rigid. \pSkip

    \item[b)] For any quadratic form $q$ on $V$, the  set $\Rig(q)$ of rigid complements $b$ in $q$ is convex in $\Quad(V).$ \pSkip

        \item[c)] If $q$ is quasilinear, then $\Rig(q)$ is a lower set in $\Quad(V)$ closed under addition.
        \end{enumerate}
        \end{thm}

        \begin{proof}
        $ $ \begin{enumerate}
        \item [a):] This is obvious from the fact that a form $\rg$ on $V$ is rigid iff $\rg(\veps_i)=0$ for every $i\in I.$ \pSkip

        \item[b):] This follows from \propref{prop:I.5.6} and the fact that in the companion matrix of $q$ the off-diagonal entries $C_{ij}, i\neq j,$ are convex in $R$ (\propref{prop:I.3.8}).
            \pSkip

        \item[c):] $\Rig(q)$ is a lower set, since $\Rig(q)$ is convex and $0\in\Rig(q).$ It follows from Propositions~\ref{prop:I.5.6} and \ref{prop:I.3.9}, that $\Rig(q)$ is closed under addition.
        \end{enumerate}
        \end{proof}

        We turn to the seemingly more subtle question of whether the convex set $\Rig(q)$ has maximal or minimal elements, and how many. In the case that $R$ is a nontrivial tangible supersemifield, we can resort to the explicit determination of the companion matrix in \S\ref{sec:I.4} to get an answer.

        \begin{thm}\label{thm:I.5.13}
        Assume that  $q$ is a quadratic form on a free module $V$
        over a nontrivial tangible supersemifield  $R$.
        \begin{enumerate}
        \item[i)] For any $\rg\in\Rig(q),$ there exists a minimal element $\rg'$ in $\Rig(q)$ with $\rg'\le\rg.$ \pSkip
        \item[ii)] If $R$ is dense, then $\Rig(q)$ has a unique minimal element.
        \pSkip
        \item[iii)] If $R$ is discrete, and $\rg_1,\rg_2$ are minimal elements of $\Rig(q),$ then $\rg_1(x)\cong_\nu \rg_2(x)$ for every $x\in V.$
            \pSkip

            \item[iv)] If $R$ is discrete then $\Rig(q)$ has a unique maximal element, while if $R$ is dense it can happen that $\Rig(q)$ has no maximal element.
            \end{enumerate}
            \end{thm}

            \begin{proof} All assertions follow from  \propref{prop:I.5.6} by a sharp look at the description of the off-diagonal entries $C_{ij}(q),$ $(i\ne j)$, of the companion matrix of $q$ via \propref{prop:I.4.10} and Theorems \ref{thm:I.4.12} and \ref{thm:I.4.13}. Note that in the delicate case that formula \eqref{eq:I.4.5} in \thmref{thm:I.4.12} comes into play, the set $C_{12}(q)=\{\bt\in R \ds |\bt\cong_\nu \al\}$ has the unique maximal element $e\al$ and all other elements of $C_{12}(q)$ are minimal, while in \thmref{thm:I.4.13}.b, Case II, where $R$ is dense, $C_{12}(q)$ has no maximal element, but has the minimal element zero.
            \end{proof}

\section{The supertropicalizations of a quadratic form}\label{sec:I.11}

Let $R$ be a ring, $V$ a free $R$-module, and $q: V\to R$ a
quadratic form. Assume further that~$U$ is a supertropical
semiring with ghost ideal $M:=eU,$ and $\vrp:R\to U$ is a
supervaluation (cf. Introduction and \cite{IzhakianKnebuschRowen2009Valuation}). (The case of primary interest that we
have in mind is that~$R$ is a field and $U$ is a supertropical
semifield, and hence $e\vrp: R\to M$ is a Krull valuation, in
multiplicative notation.) We describe a procedure to associate to
$q$ a quadratic form over~$U$ in various ways.

First we choose a base $(v_i \ds |i\in I)$ of the free $R$-module
$V.$\pSkip

Let $U^{(I)}$ denote the free $U$-module  consisting of the tuples $x=(x_i \ds |i\in
I)$ with $x_i\in U,$ almost all $x_i=0.$  It has the standard base  $(\veps_i
\ds |i\in I)$ consisting of the tuples with one coordinate $1$,  all other coordinates $0$.
%
%Let $U^{(I)}$ denote the standard \note{What does this mean?
%Perhaps the word ``standard'' is unnecessary since it is repeated
%in another 4 words} free $U$-module with standard base $(\veps_i
%\ds |i\in I).$ $U^{(I)}$ consists of the tuples $x=(x_i \ds |i\in
%I)$ with $x_i\in U,$ almost all $x_i=0,$ and the $\veps_i$ are the
%tuples with one coordinate 1, all other coordinates 0.
%
Thus
\begin{equation}\label{eq:I.11.1} x=\sum_{i\in I}
x_i\veps_i.\end{equation} We ``extend'' $\vrp: R\to U$ to a map
from $V$ to $U^{(I)},$ denoted by the same letter $\vrp,$ by the
formula $(a_i\in R)$
\begin{equation}\label{eq:I.11.2} \vrp\bigg(\sum_{i\in
I}a_iv_i\bigg):=\sum_{i\in I}\vrp(a_i)\veps_i.\end{equation}
Notice, that for $a\in R, $ $v\in V$
\begin{equation}\label{eq:I.11.3} \vrp(av)=\vrp(a)\vrp(v).\end{equation}

\begin{defn}\label{defn:I.11.1} If $B: V\times V\to R$ is a bilinear
form on the free $R$-module $V,$ we define a bilinear form
$$B^\vrp: U^{(I)}\times U^{(I)}\to U$$ on the $U$-module
$U^{(I)}$ by stating
\begin{equation}\label{eq:I.11.4}B^\vrp(\veps_i,\veps_j)=\vrp
(B(v_i,v_j))\end{equation}
 %\note{ I think $\veps$ should be $e$}
for any two indices $i,j\in I.$ We call $B^\vrp$ the
\textit{supertropicalization} (or ``\textit{stropicalization}''
for short) of $B$ under $\vrp$ with respect to the base $(v_i \ds
|i\in I)$ of $V.$\end{defn}

Thus, if $I=\{1,\dots,n\}$ and, in matrix notation,
$$B=\begin{pmatrix} \bt_{11} & \dots &\bt_{1n}\\
\vdots & \ddots & \vdots\\
\bt_{n,1}& \dots & \bt_{n,n}\end{pmatrix}$$ with
$\bt_{i,j}:=B(v_i,v_j),$ then
\begin{equation}\label{eq:I.11.5} B^\vrp =
\begin{pmatrix} \vrp(\bt_{1,1}) & \dots &\vrp(\bt_{1,n})\\
\vdots & \ddots & \vdots\\
\vrp(\bt_{n,1})& \dots &
\vrp(\bt_{n,n})\end{pmatrix}.\end{equation}

\begin{remark}\label{rem:I.11.2}
It follows from \eqref{eq:I.11.4} that for any two vectors $a,b\in
\bigcup_{i\in I} Rv_i$
\begin{equation}\label{eq:I.11.6}B^\vrp(\vrp(a),\vrp(b)) \ds =\vrp(B(a,b)).\end{equation}
For other vectors $a,b$ in $V$ this often fails.
%\note{Preferable? Replace ``this may be wrong'' by
%``\eqref{eq:I.11.6} could fail.''}
\end{remark}

 We are now ready to define the supertropicalizations
of a given quadratic form $q: V\to R$. We choose a total ordering
of the index set $I,$ and then have the unique triangular bilinear
form
$$B:=\triangledown q: V\times V\to R$$
at our disposal, which expands $q$ (cf.~\eqref{eq:I.1.22}) \footnote{All the formulas \eqref{eq:I.1.14}--\eqref{eq:I.1.24} readily generalize to infinite totally ordered bases by use of $I \times  I$-matrices instead of $n\times n $ matrices.}. It gives
us a triangular form $B^\vrp$ on $U^{(I)}.$

\begin{defn}\label{defn:I.11.3}
We define the \bfem{stropicalization}
(=\bfem{supertropicalization}) $q^\vrp:U^{(I)}\to U$ \bfem{of} $q$
\bfem{under} $\vrp$ \bfem{with respect to the ordered base} $\tL:=(v_i
\ds |i\in I)$ of $V$ by the formula
\begin{equation}\label{eq:I.11.7}
q^\vrp(x):=B^\vrp(x,x)
\end{equation}
for $x\in U^{(I)}.$ By this definition
\begin{equation}\label{eq:I.11.8}
\triangledown(q^\vrp) =(\triangledown q)^\vrp.
\end{equation}

If $I=\{1,2,\dots,n\}$ and
\begin{equation}\label{eq:I.11.9}
q=\begin{bmatrix} a_{1,1} & \dots & a_{1,n}\\
 & \ddots &\vdots\\
 &  & a_{n,n}\end{bmatrix},
\end{equation}
then
\begin{equation}\label{eq:I.11.10}
q^\vrp=\begin{bmatrix}\vrp( a_{1,1}) & \dots & \vrp(a_{1,n})\\
 & \ddots &\vdots\\
 & & \vrp(a_{n,n})\end{bmatrix}.
\end{equation}
 \end{defn}

In other terms, if we write $q$ as a polynomial with variables
$\lm_1,\dots,\lm_n,$ $$q=\sum_{i\le j}a_{i,j}\lm_i\lm_j \ds \in
R[\lm_1,\dots,\lm_n],$$ then
$$q^\vrp=\sum_{i\le j}\vrp(a_{i,j})\lm_i\lm_j \ds \in
U[\lm_1,\dots, \lm_n].$$ This means that $q^\vrp$ is the
supertropicalization of the polynomial $q$, as defined in
\cite{IzhakianRowen2007SuperTropical} and
\cite[\S9]{IzhakianKnebuschRowen2009Valuation}. More precisely,
$q^\vrp$ is the functional quadratic form on $U^{(n)}$ represented
by this polynomial.

As a consequence of Definition \ref{defn:I.11.3} and Remark
\ref{rem:I.11.2}, we have:

\begin{remark}\label{rem:I.11.4}
For any vector $a$ in $\bigcup_{i\in I} Rv_i,$ we have
$$q^\vrp(\vrp(a))=\vrp(q(a)),$$
while for other vectors in $V$ this may fail.\end{remark}

The stropicalization $q^\vrp$ comes with the balanced companion
\begin{equation}\label{eq:I.11.11}
\bt:=B^\vrp+(\trn{B})^\vrp,\end{equation} where $B:=\triangledown
q$ (cf. Example~\ref{examp:I.1.4}). For $i\ne j$, we have
\begin{equation}\label{eq:I.11.12.0}\bt(\veps_i,\veps_j)=\vrp(a_{i,j}), \end{equation}
while
\begin{equation}\label{eq:I.11.12}
\bt(\veps_i,\veps_i)=\vrp(a_{i,i})+\vrp(a_{i,i})=e\vrp(a_{i,i}).\end{equation}

\begin{remark}\label{rem:I.11.5}
This balanced companion $\bt$ can be different from the
stropicalization $b^\vrp$ of the companion $b$ of $q.$ Indeed,
\begin{equation}\label{eq:I.11.13} b^\vrp(\veps_i,\veps_j)=\bt(\veps_i,
\veps_j) = \vrp(a_{i,j})\end{equation} for $i\ne j$, but
\begin{equation}\label{eq:I.11.14} b^\vrp(\veps_i,\veps_i)=\vrp(2a_{i,i} ),
 \end{equation}
 which may very well be tangible, even if $2$ is not a unit in the
 valuation ring of $e\vrp.$
Also $b^\vrp$ is a companion of $q^\vrp$, as can easily be deduced
from Example~\ref{examp:I.2.5}, observing that $\vrp(2a_{ii})
+e\vrp(a_{ii}) = e\vrp(a_{ii})$.
\end{remark}

Let us denote the stropicalization $(U^{(I)},q^\vrp )$ of the
quadratic module $(V,q)$  with respect to the base $\tL$ more
briefly by $(\tlV, \tlq)$. The quadratic $U$-module  $(\tlV, \tlq)$
is a rather rigid object, since, by \thmref{uniq}, the base
$(\veps_i \| i \in I)$ of $\tlV$ can be only changed projectively.
% , cf. Theorem \ref{thm:I.6.11}.
 In imaginative terms, the base $\tL$ of $V$ becomes ``frozen'' in the free quadratic $U$-module $(\tlV, \tlq)$  obtained from $(V,q)$ by kind of degenerate
scalar extension $\vrp: R \to U.$ \{$\vrp$ is multiplicative, but respects addition only in a very weak way.\}

It makes sense to regard the isomorphism class of $(\tlV, \tlq)$ as an invariant of the pair $(q,\tL)$, measuring the shape of the base $\tL$ with respect to $q$ by use of the supervaluation $\vrp:R \to U$. In the special case that $eU$ is a semifield and $I = \{ 1,2\}$, i.e., $V$ is free of rank $2$, this will become explicit in \cite{QF2} via a study of pairs of vectors in quadratic $eU$-modules by kind of ``tropical trigonometry''.

Already now, in the current state of the infancy of the theory of supertropical quadratic forms, something of interest can be observed by considering the quasilinear part $\tlq_{\QL}$ of~$\tlq$ (cf.~ \S\ref{sec:I.7}).

Assuming for simplicity that $I$ is finite, $I = \{1,\dots, n \}$, let $a_i := q(v_i)$ ($1 \leq i \leq n $). Then~$\tlq_{\QL}$ can be viewed as the sequence of square classes (cf. Definition \ref{defn:I.4.4})
\begin{equation}\label{eq:I.7.16}
[\vrp(a_1)]_\sq, \ds \dots, [\vrp(a_n)]_\sq.
\end{equation}
Let us just consider the very special case that $R$ is a field of
characteristic $\neq 2$ and $\vrp$ is the initial strong
supervaluation $\brvrp_v$ covering a Krull valuation $v$ on $R$, as
described in \cite[~Theorem
10.2]{IzhakianKnebuschRowen2009Valuation}, and assume also that
$v(2) \neq 0.$ If we now choose for $\tL$ an orthogonal base of
$(V,q)$ and neglect the ordering of $I$, i.e., consider the
$S_n$-orbit  of the sequence \eqref{eq:I.7.16}, then it is easy to
interpret this orbit in classical terms, namely as the set of all
residue class forms of the $q$ with respect to the valuation $v$.
\{These are the non-zero forms $\lm_*(cq)$, with $c \in  R^*$ and
$\lm$ the canonical place $R \to \mfo_v / m_v \cup \infty$
associated to $v$, cf. \cite[Chapter 1]{Spez}.\} But for other bases
of $V$ the sequence \eqref{eq:I.7.16} seems to be something
genuinely new in quadratic form theory.

% Very roughly, the isomorphism class of the free supertropical quadratic module $(U^{(I)},q^\vrp )$ may be
% viewed as an invariant of the pair $(q,\tL)$ measuring the ``position'' of the quadratic form $q$ relative to
% the ordered base $\tL$ of the free $R$-module $V$ by the supervaluation $\vrp :R\to U$. This
% suggests itself, in view of the fact that the base $(\veps_i \ds |i\in I)$ of $U^{(I)}$ can be changed only in a very
% minor way: We only may multiply the $\veps_i$ by units of $U,$ cf. \thmref{thm:I.6.11}. In imaginative terms, the
% base $\tL$ becomes ``frozen'' in a free quadratic $U$-module obtained from~$(V,q)$ by a kind of degenerate
% scalar extension $\vrp :R\to U$. \{$\vrp $ is multiplicative, but respects addition only in a very weak
% way.\}
%
% On the other hand, while $q$ has a unique companion $b$, $q^\vrp $ may have companions different from~ $ b^\vrp ,$
% which we are free to use.
%
% Which advantages should we expect by passing from $(V,q)$ to a supertropicalization $(U^{(I)},q^\vrp )?$
% Something  can be located already at the present early stage of a theory of supertropical quadratic forms.
%
%---------------------------------------------

 The isomorphism class of $(\tlV,\tlq)$ itself  is an invariant of $(q,\tL)$, which  perhaps is too clumsy for
  practical purposes if $\dim(V)$ is big. But we can look for quotients of submodules of $\tlV$ by equivalence relations, which are
 compatible with $\tlq$ in a suitable sense, and thus gain supertropical quadratic modules to be used
 as invariants of $(q,\tL)$, which can be handled more easily.
 To give just one example: After choosing a companion $\tlb$ of $\tlq$, it turns out\footnote{To be proved in a sequel to this paper.}, that for
 any $c\in U\setminus\{0\}$ and $x\in\tlV$ the set
 $$B_c(x):=\{y\in\tlV \ds | \tlb(x,y)^2\le c\tlq(x)\tlq(y)\},$$
 is a $U$-submodule of $\tlV.$

  The family $(B_c(x) \ds |c\in R\setminus\{0\}),$
 (with fixed $x\in\tlV)$ is a kind of filtration of the quadratic module~$(\tlV,\tlq)$, which well deserves
 to be studied under this perspective.

 Note that submodules of $\tlV,$ and, all the more, quotients of these, most often are not free. Already
 this indicates the need for a supertropical quadratic form theory admitting $U$-modules of a rather general
 nature, not just free ones.

\end{document}